\documentclass[12pt]{amsart}
\usepackage[colorlinks=true,citecolor=blue,linkcolor=blue,urlcolor=black]{hyperref}
\usepackage{latexsym,amsmath,amssymb}
\usepackage{multimedia}
\usepackage{mathtools} 
\usepackage{xparse} 
\usepackage{accents}
\usepackage{color}
\usepackage{a4wide}
\usepackage{soul} 
\usepackage[colorinlistoftodos,prependcaption,textsize=tiny]{todonotes}

\title[$\epsilon$-regularity at free boundary]{Epsilon-regularity for $p$-harmonic maps at a free boundary on a sphere}
\author{Katarzyna Mazowiecka}
\author{R\'emy Rodiac}
\author{Armin Schikorra}

\address[Katarzyna Mazowiecka]{{Universit\'e catholique de Louvain, Institut de Recherche en Math\'ematique et Physique, Chemin du Cyclotron 2 bte L7.01.02, 1348 Louvain-la-Neuve, Belgium}}
\email{{katarzyna.mazowiecka@uclouvain.be}}

\address[R\'emy Rodiac]{
Universit\'e catholique de Louvain, Institut de Recherche en Math\'ematique et Physique, Chemin du Cyclotron 2 bte {L7.01.02, 1348} Louvain-la-Neuve, Belgium}
\email{remy.rodiac@uclouvain.be}

\address[Armin Schikorra]{Department of Mathematics,
University of Pittsburgh,
301 Thackeray Hall,
Pittsburgh, PA 15260, USA}
\email{armin@pitt.edu}

plus 6pt minus 12pt
\def\eps{\varepsilon}

\def\N{{\mathbb N}}
\def\H{{\mathcal H}}

\def\S{{\mathbb S}}

\def \p {\partial}

\newtheorem{theorem}{Theorem}
\newtheorem{lemma}[theorem]{Lemma}
\newtheorem{corollary}[theorem]{Corollary}
\newtheorem{proposition}[theorem]{Proposition}

\newtheorem*{remark}{Remark}
\theoremstyle{definition}
\newtheorem{definition}[theorem]{Definition}

\def\dist{{\rm dist\,}}
\def\divv{{\rm div\,}}

\def\curl{{\rm curl\,}}

\def\supp{{\rm supp\,}}

\def\Curl{{\rm Curl\,}}

\newcommand{\R}{\mathbb{R}}

\newcommand{\brac}[1]{\left (#1 \right )}

\newcommand{\Ep}{\bigwedge\nolimits}

\newcommand{\lE}[1]{{E{(#1)}}}

\newcommand{\barint}{
\rule[.036in]{.12in}{.009in}\kern-.16in \displaystyle\int }

\newcommand{\barcal}{\mbox{$ \rule[.036in]{.11in}{.007in}\kern-.128in\int $}}

\def\mvint_#1{\mathchoice
          {\mathop{\vrule width 6pt height 3 pt depth -2.5pt
                  \kern -8pt \intop}\nolimits_{\kern -3pt #1}}%
          {\mathop{\vrule width 5pt height 3 pt depth -2.6pt
                  \kern -6pt \intop}\nolimits_{#1}}%
          {\mathop{\vrule width 5pt height 3 pt depth -2.6pt
                  \kern -6pt \intop}\nolimits_{#1}}%
          {\mathop{\vrule width 5pt height 3 pt depth -2.6pt
                  \kern -6pt \intop}\nolimits_{#1}}}

\numberwithin{theorem}{section} \numberwithin{equation}{section}

\newcommand{\lap}{\Delta }
\newcommand{\laph}{\laps{1}}
\newcommand{\aleq}{\precsim}

\newcommand{\aeq}{\approx}

\newcommand{\laps}[1]{(-\lap) ^{\frac{#1}{2}}}

\newcommand{\lapms}[1]{I^{#1}}

\let\latexchi\chi
\makeatletter
\renewcommand\chi{\@ifnextchar_\sub@chi\latexchi}
\newcommand{\sub@chi}[2]{
  \@ifnextchar^{\subsup@chi{#2}}{\latexchi^{}_{#2}}%
}
\newcommand{\subsup@chi}[3]{
  \latexchi_{#1}^{#3}%
}
\makeatother

\let\latexeta\eta
\makeatletter
\renewcommand\eta{\@ifnextchar_\sub@eta\latexeta}
\newcommand{\sub@eta}[2]{
  \@ifnextchar^{\subsup@eta{#2}}{\latexeta^{}_{#2}}%
}
\newcommand{\subsup@eta}[3]{
  \latexeta_{#1}^{#3}%
}
\makeatother

\renewcommand{\div}{\operatorname{div}}
\newcommand{\dv}{\operatorname{div}}

\newcommand{\solu}{u}
\newcommand{\solw}{w}

\newcommand{\solv}{v}

\begin{document}

\sloppy
\begin{abstract}
We prove an $\epsilon$-regularity theorem for vector-valued $p$-harmonic maps, which are critical with respect to a partially free boundary condition, namely that they map the boundary into a round sphere.

This does not seem to follow from the reflection method that Scheven used for harmonic maps with free boundary (i.e., the case $p=2$): the reflected equation can be interpreted as a $p$-harmonic map equation into a manifold, but the regularity theory for such equations is only known for round targets.

Instead, we follow the spirit of the last-named author's recent work on free boundary harmonic maps and choose a good frame directly at the free boundary. This leads to growth estimates, which, in the critical regime $p=n$, imply H\"older regularity of solutions. In the supercritical regime, $p < n$, we combine the growth estimate with the geometric reflection argument: the reflected equation is super-critical, but, under the assumption of growth estimates, solutions are regular.

In the case $p<n$, for stationary $p$-harmonic maps with free boundary, as a consequence of a monotonicity formula we obtain partial regularity up to the boundary away from a set of $(n-p)$-dimensional Hausdorff measure.
\end{abstract}

\maketitle
\tableofcontents
\section{Introduction}
Over the last few years the theory of half-harmonic maps received a lot of attention, beginning with the pioneering work of Da Lio and Rivi\`{e}re \cite{DaLio-Riviere-2011,DaLio-Riviere-2011-2}, see also the subsequent \cite{Schikorra-2012, DaLio-2013, Millot-Sire-2015, Schikorra-2015}. Half-harmonic maps appear in nature as free boundary problems --- e.g., they are connected to critical points of the energy
\[
 \|\nabla u\|^2_{L^2(D,\R^N)} \quad \mbox{s.t. $u(\partial D) \subset {\mathcal{N}}$ in the a.e. trace sense}.
\]
Here, $D\subset\R^n$ {is an open set} and ${\mathcal{N}} \subset \R^N$ is a smooth closed manifold. The Euler-Lagrange equations of the latter problem are
\begin{equation}\label{eq:freebdharm}
\begin{cases}
\lap u = 0 \quad &\mbox{in $D$}\\
\partial_\nu \solu\perp T_\solu {\mathcal{N}} \quad &\mbox{on $\partial D$},\\
\end{cases}
\end{equation}
where $\nu$ denotes the outer normal vector.

For $D = \R^n_+$ and $\partial D = \R^{n-1} \times \{0\}$ the equation \eqref{eq:freebdharm} is equivalent to
\begin{equation}\label{eq:halfharm}
\begin{cases}
\lap u = 0 \quad &\mbox{in $\R^n_+$}\\
\laph_{\R^{n-1}} \solu\perp T_\solu {\mathcal{N}} \quad &\mbox{on $\R^{n-1}\times \{0\}$}.\\
\end{cases}
\end{equation}
Here, $\laph_{\R^{n-1}}$ denotes the half-Laplacian acting on functions defined on $\R^{n-1} \times \{0\}$. The equation $\laph_{\R^{n-1}} \solu\perp T_\solu {\mathcal{N}}$ is the half-harmonic map equation, for an overview see \cite{DaLio-Riviere-2011}.

The equivalence of \eqref{eq:freebdharm} and \eqref{eq:halfharm} is crucially related to the fact that we are considering critical points of an $L^2$-energy. Several notions of fractional $p$-harmonic maps have been proposed. In \cite{DaLio-Schikorra-2014,DaLio-Schikorra-2017} Da Lio and the third-named author considered $H^{s,p}$-harmonic maps, i.e., critical points of 
\begin{equation}\label{eq:Hsharm}
 \|\laps{s} u\|^p_{L^p(\R^{n-1},\R^N)} \quad \mbox{s.t. $u(x) \in {\mathcal{N}}$ for a.e. $x\in\R^{n-1}$}.
\end{equation}
In \cite{Schikorra-2015Lp} energies with a gradient-type structure were studied, namely 
\begin{equation}\label{eq:gradsharm}
 \|D^s u\|^p_{L^p(\R^{n-1},\R^N)} \quad \mbox{s.t. $u(x) \in {\mathcal{N}}$ for a.e. $x\in\R^{n-1}$},
\end{equation}
where $D^s = D \lapms{1-s}$ is the Riesz-fractional gradient, see also \cite{Shieh-Spector,Shieh-Spector2}. Finally, $W^{s,p}$-harmonic maps were studied in \cite{SchikorraCPDE}, that is critical points of the energy
\begin{equation}\label{eq:wspharm}
 \int_{\R^{n-1}}\int_{\R^{n-1}} \frac{|u(x)-u(y)|^p}{|x-y|^{n+sp}}\ dx\ dy \quad \mbox{s.t. $u(x) \in {\mathcal{N}}$ for a.e. $x\in\R^{n-1}$},
\end{equation}
see also \cite{Mazowiecka-Schikorra-2017}. All these versions of fractional $p$-harmonic maps have one thing in common: they do not seem related to a free boundary equation \eqref{eq:freebdharm}. For \eqref{eq:Hsharm} and \eqref{eq:gradsharm} this is clear, since the energies {are defined on} the ``wrong'' function space $H^{s,p}$. {Indeed, a map in } $W^{1,p}(D)$ {has a trace in} $W^{1-\frac{1}{p},p}(\partial D)$, {but} $W^{1-\frac{1}{p},p}(\partial D) \neq H^{1-\frac{1}{p},p}(\partial D)$ for $p \neq 2$. For the $W^{s,p}$-energy \eqref{eq:wspharm} it is an interesting open problem if it is possible to find a $p$-harmonic extension that interprets this problem as a free boundary problem.

In this work we concentrate on free boundary problems. {We focus on smooth bounded domains, so in the sequel $D$ is such a domain.} We prove regularity at the free boundary for critical points $u:D \to \R^N$ of the energy 
\begin{equation}\label{eq:penergy}
 \|\nabla u\|^p_{L^p(D,\R^N)} \quad \mbox{s.t. $u(\partial D) \subset {\mathcal{N}}$ in the a.e. trace sense}.
\end{equation}
It is not clear that the space $\mathcal{A}:=\{u \in W^{1,p}(D,\R^N)\colon \ u(\partial D) \subset {\mathcal{N}}\}$ possesses a natural structure of a smooth Banach manifold. That is why we shall define what we mean by critical point.
\begin{definition}
We say that $u$ is a critical point of $\int_D |\nabla u|^p$ in the space $\mathcal{A}$ if $u$ satisfies
\begin{equation}\label{eq:def}
{\int_D |\nabla u|^{p-2}\nabla u \cdot \nabla \phi=0}
\end{equation}
for all $\phi$ in $W^{1,p}(D,\R^N)$ s.t. its trace $\phi(x)\rvert_{\partial D}$ is in $T_{u(x)}{\mathcal{N}}$ a.e. Such a critical point is called a \textit{$p$-harmonic map with free boundary}.
\end{definition}
Equation \eqref{eq:def} is obtained by requiring that for every $C^{1}$-path $\gamma:(-1,1)\rightarrow \mathcal{A}$ such that $\gamma(0)=u$ we have
\begin{equation}
\frac{d}{dt}\bigg\rvert_{t=0} \int_D |\nabla \gamma(t)|^p=0.
\end{equation}
\begin{remark}
{Although this is not relevant for our purpose, let us remark that equation} \eqref{eq:def} {can be interpreted} as $u$ satisfying {in a distributional sense}
\begin{equation}\label{eq:soludef}
\begin{cases}
\dv (|\nabla \solu|^{p-2} \nabla \solu) = 0 \quad &\mbox{in $D$}\\
|\nabla \solu|^{p-2}\partial_\nu \solu\perp T_\solu {\mathcal{N}} \quad &\mbox{on $\partial D$}.
\end{cases}
\end{equation}
{Note that, by definition, $\solu$ is a solution of \eqref{eq:soludef} in the sense of distributions if and only if 
\begin{equation}\label{eq:soludef2}
\int_D |\nabla \solu|^{p-2}\nabla \solu\cdot \nabla \phi =0
\end{equation}
for all $\phi \in C^\infty(\overline{D},\R^N)$ with $\phi(x)\in T_{\solu(x)}\mathcal{N}$ for $\mathcal{H}^{n-1}$-a.e.\ $x \in \p D$. Indeed, taking $\phi \in C^\infty_c(D,\R^N)$ we obtain the interior equation 
$$\div( |\nabla \solu|^{p-2}\nabla \solu)=0 \text{ in } D.$$
As for the boundary equation, we can see that if $\solu$ is smooth enough and satisfies \eqref{eq:soludef2} then after an integration by parts we find
\begin{equation}\label{eq:soludefboundary}
\int_{\p D} |\nabla \solu|^{p-2}\p_\nu \solu \cdot \varphi=0.
\end{equation}
Since any $\varphi \in C^\infty(\p D,\R^N)$ with $\varphi(x) \in T_{u(x)}\mathcal{N}$ can be extended in a function $\phi \in C^\infty(\overline{D},\R^N)$, thus \eqref{eq:soludefboundary} implies
\[
|\nabla \solu|^{p-2}\p_\nu u \perp T_\solu \mathcal{N} \text{ on } \p D.
\]
The equivalence between being a solution of \eqref{eq:soludef} in the sense of distributions and being a critical point of the $p$-energy in the space $\mathcal{A}$ is {true if $u$ is smooth enough, for example $u \in C^1(\overline{D},\R^n)$ is sufficient. Indeed, in this case we can see that we have  density of 
$\{ \phi \in C^\infty(\overline{D},\R^N)\colon \phi \in T_{\solu} \mathcal{N}\}$ in $\{\phi \in W^{1,p}(D,\R^N)\colon \phi\big\rvert_{\partial D}\in T_\solu\mathcal{N}\}$.
}}
\end{remark}

%
%
The natural starting point, when studying equations of the form \eqref{eq:soludef}, is the regularity theory. The interior regularity is known and follows from the interior equation and results of  \cite{Uhlenbeck-1977,tolksdorf1984regularity}, see also the recent \cite{Kuusi-Mingione-2017}. Hence, the main difficulty is the regularity up to the boundary.  For an arbitrary manifold ${\mathcal{N}}$ a regularity theory for a solution \eqref{eq:soludef} is out of reach: even the regularity theory for the interior problem
\[
 \dv (|\nabla \solu|^{p-2} \nabla \solu) \perp T_\solu {\mathcal{N}}
\]
is known only for homogeneous targets ${\mathcal{N}}$, see Fuchs \cite{Fuchs-1993}, Takeuchi \cite{Takeuchi-1994}, Toro and Wang \cite{ToroWang-1995}, Strzelecki \cite{Strzelecki-1994,Strzelecki-1996}, and also the recent survey~\cite{Schikorra-Strzelecki}. For this reason we shall restrict our attention to the sphere $\S^{N-1} \subset \R^N$. In the rest of the paper we consider the problem:
\begin{equation}\label{eq:soluspheredef}
\begin{cases}
\dv (|\nabla \solu|^{p-2} \nabla \solu) = 0 \quad &\mbox{in $D$}\\
|\nabla \solu|^{p-2}\partial_\nu \solu\perp T_\solu \S^{N-1} \quad &\mbox{on $\partial D$}\\
u(\partial D)\subset \S^{N-1}.
\end{cases}
\end{equation}
We remark that the free boundary conditions can be viewed as boundary conditions mixed between Dirichlet and homogeneous Neumann boundary conditions. Indeed, in the sphere case we have a Dirichlet boundary condition for the norm of $u$: $|u|=1$ on $\partial D$ and homogeneous Neumann condition for the ``phase" $\partial_\nu \left(\frac{u}{|u|}\right)=0$. To see that in the case of a general manifold we can use Fermi coordinates near some points of ${\mathcal{N}}$ as explained in \cite[p.938-939]{Fraser2000} in the context of minimal surfaces with free boundaries (for more on minimal surfaces with free boundaries see also \cite{FraserSchoen2013} and the references therein).

Our main theorem is the following $\epsilon$-regularity type theorem.
\begin{theorem}[$\epsilon$-regularity]\label{th:main}
Let $D \subset \R^n$ be a smooth, bounded domain and $p\geq 2$. Then there exist $\epsilon = \epsilon(p,n,D)>0$ and $\alpha = \alpha(p,n,D)> 0$, such that for any $\solu \in W^{1,p}(D,\R^N)$ solution to \eqref{eq:soluspheredef} the following holds:
If for some $R > 0$ and for some $x_0 \in \overline{D}$
\begin{equation}\label{eq:thmain:smallnesscond}
 \sup_{|y_0-x_0| < R}\ \sup_{\rho < R} \rho^{p-n} \int_{B(y_0,\rho)\cap D} |\nabla \solu|^p < \epsilon,
\end{equation}
then $\solu$ and $\nabla \solu$ are H\"older continuous in $B(x_0,R/2) \cap \overline{D}$. Moreover, we have the following estimates:
\[
 \sup_{x, y \in B(x_0,R/2)}  \frac{|u(x)-u(y)|}{|x-y|^{\alpha}} \aleq R^{-\alpha} \brac{\sup_{|y_0-x_0| < R}\ \sup_{\rho < R} \rho^{p-n} \int_{B(y_0,\rho)\cap D} |\nabla \solu|^p}^{\frac{1}{p}}
\]
and
\[
 \sup_{x, y \in B(x_0,R/2)}  \frac{|\nabla u(x)-\nabla u(y)|}{|x-y|^{\alpha}} \aleq R^{-\alpha-1} \brac{\sup_{|y_0-x_0| < R}\ \sup_{\rho < R} \rho^{p-n} \int_{B(y_0,\rho)\cap D} |\nabla \solu|^p}^{\frac{1}{p}}.
\]
\end{theorem}
When $p=n$ this $\epsilon$-regularity implies directly (from the absolute continuity of the Lebesgue integral) that $n$-harmonic maps with free boundary {and their gradients} are H\"older continuous.
\begin{corollary}
Let $u$ and $\alpha$ be as in Theorem \ref{th:main} with $p=n$ then $u$ is in $C^{1,\alpha}(\overline{D},\R^N)$.
\end{corollary}
As usual, an $\epsilon$-regularity result such as Theorem~\ref{th:main} implies partial regularity for \emph{stationary} $p$-harmonic maps {with free boundary} (cf.\ \eqref{eq:firstvariation} for the definition).
\begin{theorem}[partial regularity]\label{th:partialregularity}
Let $D \subset \R^n$ be a smooth, bounded domain, $p\geq 2$, and assume that $\solu \in W^{1,p}(D,\R^N)$, with trace $u \in W^{1-\frac{1}{p},p}(\partial D, \S^{N-1})$, is a \emph{stationary} point of the energy \eqref{eq:penergy} {with free boundary}.
Then there exists a closed set $\Sigma \subset \overline{D}$ such that $\mathcal{H}^{n-p}(\Sigma) {=0}$ and $u \in C^{1,\alpha}(\overline{D} \backslash \Sigma)$, where $\alpha>0$ is from Theorem \ref{th:main}.
\end{theorem}
{\begin{remark}
        Although some of our results work for unbounded domains we note that finite energy, stationary $p$-harmonic maps with with free boundary satisfy a Liouville type theorem, cf. Proposition \ref{pr:Liouvilletype}. This is why we focus on bounded domains.
       \end{remark}
}
Moreover, besides giving regularity in the case $p=n$ and partial regularity in the case $p<n$, an $\epsilon$-regularity could be useful to describe the possible loss of compactness of sequences of $n$-harmonic maps with free boundaries and an energy decomposition theorem. In the case $p=n=2$, i.e., for harmonic maps with free boundaries such a result was proven in \cite{DaLio2015,LaurainPetrides2017}. Our case requires completely different methods, due to the nonlinearity of the $p$-Laplacian for $p \neq 2$.

Let us comment on our strategy for the proof of Theorem~\ref{th:main}. The natural first attempt to prove a result like Theorem~\ref{th:main} is to adapt the beautiful geometric reflection method that Scheven used in \cite{Scheven-2006} to obtain an $\epsilon$-regularity result up to the free boundary for harmonic maps, i.e., for the case $p=2$ (see also \cite{BerlyandMironescu2004} where the authors also devised a reflection technique to prove regularity up to the boundary of solutions of some Ginzburg-Landau equations with free boundary conditions). This way, one would hope to be able to rewrite the Neumann condition at the boundary into an interior equation. For $p=2$ the reflected equation has again the structure of a harmonic map (with a new metric in the reflected domain). Thus, the regularity theory for harmonic maps with a free boundary follows from the interior regularity for harmonic maps developed by H\'elein \cite{Helein-1991}, see also \cite{Riviere-2007}. For $p > 2$ there is a major drawback to that strategy: as mentioned above, the regularity theory for the interior $p$-harmonic map equation is only understood for round targets. It was not clear to us, how to interpret the reflected equation as a map into such a round target. The reflection, which generates a somewhat ``unnatural metric'' seems to destroy our boundary sphere-structure. Indeed, up to now, only the regularity theory for \emph{minimizing} $p$-harmonic maps with free boundary was understood, see \cite{Duzaar-Gastel-1998,Mueller-2002} {where it is shown that such a map is in $C^{1,\alpha}$, for some $\alpha$, outside a singular set $\mathcal{S}$ with $\text{dim}_\mathcal{H}(\mathcal{\mathcal{S}})=n-\lfloor p \rfloor-1$ and $\mathcal{S}$ is discrete if $n-1\leq p<n$}. For $p=2$ free boundary problems for \emph{minimizing} harmonic maps were studied in \cite{Duzaar-Steffen-1989,Hardt-Lin-1989}.

In this work we follow in spirit the recent work of the third-named author \cite{Schikorra-2017} which does not use a reflection technique, but rather computes an equation along the free boundary and applies a moving frame technique to this free boundary part of the equation itself. This strategy leads to \emph{growth} estimates, Proposition~\ref{pr:growth}, which for the critical case $n = p$ implies directly H\"older regularity of solutions. Once the growth estimates are established we can apply the reflection. Since the reflection is explicit, it is easy to see that the growth estimates still hold for the reflected solution, which we shall call $v$. Now $v$ solves a critical or super-critical equation of the form
\[
 |\dv (|\nabla v|^{p-2} \nabla v)| \aleq |\nabla v|^p.
\]
In principle, solutions to this equation may be singular, e.g.,  $x/|x|$ or $\log \log 1/|x|$. But with the growth estimates from Proposition~\ref{pr:growth}, which transfers to $v$, one can employ a blow-up argument due to \cite{Hardt-Kinderlehrer-Lin-1986,Hardt-Lin-1987} and then bootstrap for higher regularity.

The outline of the paper is as follows: In Section~\ref{s:growth} we state and prove the crucial growth estimate for solutions to {\eqref{eq:soluspheredef}}. In Section~\ref{s:pnHoelder} we show how this implies H\"older continuity of solutions for the case $p=n$. For $p < n$ we show in Section~\ref{s:genericHoelder} how a generic super-critical system implies H\"older regularity of solutions once the growth estimates from Proposition~\ref{pr:growth} are guaranteed. Combining this with Scheven's reflection argument, we give in Section~\ref{s:proofthmain} the proof of Theorem~\ref{th:main}. Finally, in Section~\ref{s:partialregularity}, we prove the partial regularity of solutions, i.e., Theorem~\ref{th:partialregularity}.

{\textbf {Notation.}} We denote by $B(x,r)$ the ball of radius $r$ centered at $x\in\R^n$. We write $\R^n_+=\R^{n-1}\times(0,\infty)$, $\R^n_-=\R^{n}\times(-\infty,0)$, and $B^+(x,r) = B(x,r)\cap\R^n_+$. By $(u)_\Omega$ we denote the mean value of a map $u$ on a set $\Omega$, i.e., $(u)_\Omega= \frac{1}{|\Omega|}\int_\Omega u $.

\section{The growth estimates}\label{s:growth}
{{Recall that} we assume that $D$ is a bounded set with a smooth boundary. In view of the Lemma~\ref{la:boundedness} we know that $|u| \leq 1$ holds for any solution to \eqref{eq:soluspheredef}. The arguments can be also extended to unbounded domains like $\R^n_+$ under the assumption that $u \in L^\infty_{loc}(\R^n_+)$, cf. Lemma \ref{la:boundedness2}. Note that in principle, the constants may depend on the $L^\infty$-norm of $u$.}

The main result in this section, and the crucial argument in this work, is the following growth estimate that one could interpret as a kind of Caccioppoli type estimate. We were not able to obtain such an estimate by a geometric reflection argument, since that reflection changes the metric, and only in the case of round targets, such as the sphere, regularity theory (and in particular the related growth estimates) are known.

\begin{proposition}[Growth estimates]\label{pr:growth}
Let $p\ge2$. There exists a radius $R_0$ depending only on $\partial D$ such that for any $u\in W^{1,p}(D,\R^N)$ satisfying \eqref{eq:soluspheredef} the following holds:

Whenever $B(x_0,R) \subset \R^n$, $R \in (0,R_0)$ is such that for some $\lambda \in (0,\infty)$ it holds
\begin{equation}\label{small}
 \sup_{B(y_0,r)\subset  B(x_0,R)} r^{p-n} \int_{B(y_0,r)\cap D} |\nabla \solu|^p < \lambda^p,
\end{equation}
then for any $B(y_0,4r)\subset B(x_0,R)$ and any $\mu > 0$,
\begin{equation}\label{eq:growth:peqn}
\int_{B(y_0,r)\cap D} |\nabla u|^p \leq C\ \brac{\lambda + \mu^{p-1}} \int_{B(y_0,4r)\cap D} |\nabla u|^p + C\mu^{-1} \int_{(B(y_0,4r)\backslash B(y_0,r))\cap D} |\nabla u|^p.
\end{equation}
Alternatively, we have the following estimates:

If $B(y_0,2r) \backslash D = \emptyset$, then
\begin{equation}\label{eq:growth:pllnbd0}
 \int_{{B(y_0,r)}} |\nabla u|^p \leq C \lambda \int_{B(y_0,4r)\cap D} |\nabla u|^p + C \lambda^{1-p} r^{-p}\int_{B(y_0,4r)\cap D} |\solu-(\solu)_{B(y_0,4r){\cap D}}|^p.
\end{equation}
If $B(y_0,2r) \backslash D \neq \emptyset$, then
\begin{equation}\label{eq:growth:pllnbd}
\begin{split}
 \int_{B(y_0,r)\cap D} |\nabla u|^p &\leq C \lambda \int_{B(y_0,4r)\cap D} |\nabla u|^p + C \lambda^{1-p} r^{-p}\int_{B(y_0,4r)\cap D} |\solu-(\solu)_{B(y_0,4r){\cap D}}|^p\\
 &\quad+ C\lambda^{1-p}r^{-p}\int_{B(y_0,4r)\cap D} ||\solu|^2-1|^p,
\end{split}
 \end{equation}
for a constant $C=C(n,p,D)$.
\end{proposition}
%
Our strategy, in principle, is to adapt the method for harmonic maps into spheres developed by H\'elein \cite{Helein-1990}, see Strzelecki's \cite{Strzelecki-1994} for the $n$-harmonic case. To motivate our approach, we briefly outline their strategy for a $p$-harmonic map $w\in W^{1,p}(D,\S^{N-1})$, i.e., a solution to
\begin{equation}\label{eq:pharmmapeq}
 \div(|\nabla w|^{p-2} \nabla w) \perp T_w \S^{N-1}.
\end{equation}
The first step is to rewrite this equation. Since $w \in \S^{N-1}$ we have $w \in \brac{T_w \S^{N-1}}^\perp$. Consequently, \eqref{eq:pharmmapeq} can be rewritten in distributional sense as
\begin{equation}\label{eq:pharmmapeq2}
 \int_{D} |\nabla w|^{p-2} \nabla w^i\cdot \nabla \varphi = \int_{D} |\nabla w|^{p-2} \nabla w^k \cdot \nabla (w^k w^i \varphi), 
\end{equation}
which holds for all $\varphi \in C_c^\infty(D)$ and $i = 1,\ldots,N$. Here and henceforth, we use the summation convention.

Next, from $|w| \equiv 1$, we get $w^k \nabla w^k  \equiv \frac{1}{2} \nabla |w|^2 = 0$. Consequently, \eqref{eq:pharmmapeq2} can be written as
\begin{equation}\label{eq:pharmmapeq3}
 \int_{D} |\nabla w|^{p-2} \nabla w^i\cdot \nabla \varphi = \int_{D} |\nabla w|^{p-2} \nabla w^k \cdot \brac{\nabla w^k\ w^i - \nabla w^i\ w^k} \varphi.
\end{equation}
Now one observes that from \eqref{eq:pharmmapeq2} a conservation law follows, a fact that for $p=n=2$ was discovered by Shatah \cite{Shatah-1988},
\begin{equation}\label{eq:pharmmapeq3conslaw}
 \div\brac{|\nabla w|^{p-2} \brac{\nabla w^k\ w^i - \nabla w^i\ w^k} } = 0 \quad \mbox{in $D$}.
\end{equation}
Thus, $|\nabla w|^{p-2} \nabla w^k \cdot \brac{\nabla w^k\ w^i - \nabla w^i\ w^k}$ is a div-curl term and with the help of the celebrated result of Coifman, Lions, Meyer, and Semmes, \cite{CLMS}, one obtains a growth estimate.

The above argument heavily relied on the fact that $w^k \nabla w^k \equiv 0$. It is important to observe that this trick will not work in the situation from Theorem~\ref{th:main}: if we only know that $u\Big|_{\partial D} \subset \S^{N-1}$, then there is no reason that $u \cdot \nabla u = 0$ in $D$.
Nevertheless, we will stubbornly follow the strategy outlined above, just along the boundary $\partial D$, keeping the extra terms that involve $u^k \nabla u^k$. Firstly, we find:
\begin{lemma}\label{la:testfunctions}
For $\solu\in W^{1,p}(D,\R^N)$ satisfying \eqref{eq:soluspheredef} we have
\[
 \int_{D} |\nabla \solu|^{p-2} \nabla \solu^i\cdot \nabla \varphi = \int_{D} |\nabla \solu|^{p-2} \nabla \solu^k \cdot \nabla (\solu^k \solu^i \varphi),
\]
for any $\varphi \in W^{1,p}(D)$.
\end{lemma}
Let us stress that the test function $\varphi$ above does not need to vanish at the boundary.
\begin{proof}
Let $\Phi= (0,\ldots,\varphi,\ldots,0)$ (only the i-th coordinate is nonzero and equal to $\varphi$). Observe that 
\[\Phi - \solu \langle \solu,  \Phi\rangle_{\R^N} \in T_\solu \S^{N-1} \quad \mbox{a.e. on $\partial D$.}
                                                 \]
The claim follows now from the definition of $p$-harmonic maps with free boundary \eqref{eq:def}.
\end{proof}
Also we have the following conservation law.
\begin{lemma}\label{la:conservation}
Let $\solu\in W^{1,p}(D,\R^N)$ satisfying \eqref{eq:soluspheredef}. Then, for
\[
 \Omega_{ij} := \brac{\solu^i\nabla \solu^j -\solu^j\nabla \solu^i},
\]
we have 
\[
 \div(|\nabla \solu|^{p-2} \Omega_{ij}) = 0 \quad \mbox{in $D$}
\]
up to the boundary. That is, for any $\varphi \in C^\infty(\overline{D})$ and any $i,j = 1,\ldots,N$,
\begin{equation}\label{eq:conservation}
 \int_{D} |\nabla \solu|^{p-2} \Omega_{ij}\cdot \nabla \varphi = 0.
\end{equation}
Besides, equation \eqref{eq:conservation} is also satisfied for every $\varphi$ in $W^{1,p}\cap L^\infty(D)$.
\end{lemma}
\begin{proof}
By the product rule, 
\[
\begin{split}
  \int_{D}& \nabla \varphi \cdot |\nabla \solu|^{p-2}\brac{\solu^i\nabla \solu^j -\solu^j\nabla \solu^i}\\
  &= \int_{D} \brac{\nabla (\varphi \solu^i) \cdot |\nabla \solu|^{p-2}\nabla \solu^j - \nabla (\varphi \solu^j) \cdot |\nabla \solu|^{p-2}\nabla \solu^i}.
\end{split}
 \]
Therefore, by Lemma~\ref{la:testfunctions}, we find
\[
\begin{split}
 \int_{D}& |\nabla \solu|^{p-2} \Omega_{ij}\cdot \nabla \varphi\\
 &=\int_{D} |\nabla \solu|^{p-2} \nabla \solu^k \cdot \nabla (\solu^k \solu^i \solu^j\varphi) -\int_{D} |\nabla \solu|^{p-2} \nabla \solu^k \cdot \nabla (\solu^k \solu^j \solu^i \varphi) \\
 &= 0.
 \end{split}
\]
\end{proof}
We combine Lemma~\ref{la:conservation} and Lemma~\ref{la:testfunctions}. In contrast to the argument for the $p$-harmonic map $w$, we find additional terms. Namely, instead of having $w^k \nabla w^k \equiv 0$ we merely have $u^k \nabla u^k = \frac{1}{2} \nabla (|u|^2-1)$. However, it is an improvement, because $|u|^2-1 \in W^{1,p}_0(D)$.
\begin{lemma}\label{la:goodpde}
Let $\solu\in W^{1,p}(D,\R^N)$ satisfying \eqref{eq:soluspheredef}. Then for any $\varphi \in W^{1,p}(D)$ we have
\[
\begin{split}
 \int_{D}& |\nabla \solu|^{p-2} \nabla \solu^i\cdot \nabla \varphi\\
 &= \int_{D} |\nabla \solu|^{p-2} \nabla \solu^k \cdot \Omega_{ik}\, \varphi+ \int_{D} |\nabla \solu|^{p-2} \nabla \solu^i\cdot \nabla \brac{|\solu|^2-1} \ \varphi\\
 &\quad +\frac{1}{2}\int_{D} |\nabla \solu|^{p-2} \nabla \varphi \cdot \nabla \brac{|\solu|^2-1}\ \solu^i.
\end{split}
 \]
 \end{lemma}
It is important to observe that in particular we do not obtain an equation of the form $|\div(|\nabla u|^{p-2} \nabla u)| \aleq |\nabla u|^p$ as it is the case for $p$-harmonic maps (i.e., the interior situation). This is why for $p < n$ we are forced to combine our growth estimate with the geometric reflection argument, see Proposition~\ref{pr:plapnup}.
\begin{proof}[Proof of Lemma~\ref{la:goodpde}]
By Lemma~\ref{la:testfunctions} we have for any $\varphi \in C^\infty(\overline{D})$,
\[
\begin{split}
 \int_{D}& |\nabla \solu|^{p-2} \nabla \solu^i\cdot \nabla \varphi \\
 &= \int_{D} |\nabla \solu|^{p-2} \nabla \solu^k \cdot \nabla \solu^k\ \solu^i \varphi +\int_{D} |\nabla \solu|^{p-2} \nabla \solu^k \cdot \solu^k \nabla(\solu^i \varphi).
 \end{split}
 \]
Using the definition of $\Omega_{ik}$ from Lemma~\ref{la:conservation} we write
\[
\begin{split}
 \int_{D}& |\nabla \solu|^{p-2} \nabla \solu^i\cdot \nabla \varphi \\
 &= \int_{D} |\nabla \solu|^{p-2} \nabla \solu^k \cdot \Omega_{ik}\, \varphi+ 2\int_{D} |\nabla \solu|^{p-2} \nabla \solu^i\cdot \nabla \solu^k  \solu^k \ \varphi\\
 &\quad +\int_{D} |\nabla \solu|^{p-2} \nabla \solu^k \solu^i \solu^k \cdot \nabla \varphi.
\end{split}
 \]
Since $\solu^k \nabla \solu^k  = \frac{1}{2}\nabla \brac{|\solu|^2-1}$, we have shown that
\[
\begin{split}
 \int_{D}& |\nabla \solu|^{p-2} \nabla \solu^i\cdot \nabla \varphi\\
 &= \int_{D} |\nabla \solu|^{p-2} \nabla \solu^k \cdot \Omega_{ik}\, \varphi+ \int_{D} |\nabla \solu|^{p-2} \nabla \solu^i\cdot \nabla \brac{|\solu|^2-1} \ \varphi\\
 &\quad +\frac{1}{2}\int_{D} |\nabla \solu|^{p-2} \nabla \varphi \cdot \nabla \brac{|\solu|^2-1}\ \solu^i  \cdot \\
\end{split}
 \]
\end{proof}
For the second and third term on the right-hand side of the equation in Lemma~\ref{la:goodpde} we observe that $|u|^2-1$ has zero boundary values on $\partial D$. In addition, and this is another crucial ingredient here, we can choose $u$ or (its coordinates) as a test function in Lemma \ref{la:testfunctions}, \ref{la:conservation}, and \ref{la:goodpde} since $u$ is in $W^{1,p}\cap L^\infty(D,\R^N)$ from Lemma \ref{la:boundedness}.


 Moreover, in view of the interior equation for $u$, \eqref{eq:soludef},
\[
 \int_{D} |\nabla u|^{p-2} \nabla u^i\cdot \nabla (|u|^2-1) = 0.
\]

\begin{proof}[Proof of Proposition~\ref{pr:growth}]
For notational simplicity {we prove the growth estimates when the boundary is flat. More precisely we treat the case where  $B^+(0,R)\subset D \subset \R^n_+$ for some $R>0$, and  $\partial D\cap B(0,R)=\partial \R^n_+ \cap B(0,R)$.} 
The following argument can be easily adapted to general $D$ --- here is where one has to choose $R_0 = R_0(D)$ for flattening the boundary. We leave the details to the reader. {We also recall that, since we work in a smooth bounded domain, from Lemma~\ref{la:boundedness} we {have} that $\|u\|_{L^\infty(D)}\leq 1$.}


Let $\eta \in C_c^\infty(B(0,2))$ be the typical bump function constantly one in $B(0,1)$. {Let $y_0\in \R^n, r>0$ be such that $B(y_0,4r)\subset B(0,R)$}. Denote by
\[
 \eta_{B(y_0,r)}(x) := \eta((x-y_0)/r).
\]
Set \[\tilde{\solu} := \eta_{B(y_0,r)} (\solu-(\solu)_{B^{{+}}(y_0,2r)})\] and
\[
 \hat{\solu} := (1-\eta_{B(y_0,r)})\eta_{B(y_0,r)} (\solu-(\solu)_{B^{{+}}(y_0,2r)}).
 \]
Since $\eta_{B(y_0,r)} \equiv 1$ on $B(y_0,r)$ we have 
\[
 \int_{ B^+(y_0,r)}|\nabla \solu|^p \leq \int_{{\R^n_+}} |\nabla \solu|^{p-2} \nabla \tilde{\solu} \cdot \nabla \tilde{\solu}.
\]
We compute
\begin{equation}\label{eq:gradientutildauwidehat}
\begin{split}
 \nabla \tilde{\solu} \cdot \nabla \tilde{\solu}
 &= \nabla \solu\cdot \nabla \tilde{\solu}-  \nabla \solu\cdot \nabla \hat{\solu} \\
 &\quad- \nabla \eta_{B(y_0,r)}\cdot \nabla \solu\ \tilde{\solu}  + \nabla\eta_{B(y_0,r)}\, (\solu-(\solu)_{B^{{+}}(y_0,2r)})\cdot \nabla \tilde{\solu}.
\end{split}
 \end{equation}
Since $|\nabla \eta_{B(y_0,r)}| \aleq r^{-1}$, 
\begin{equation}\label{eq:thirdterm}
 \int_{{\R^n_+}}|\nabla \solu|^{p-2} (\nabla \eta_{B(y_0,r)}\tilde{\solu})\cdot \nabla \solu\  \aleq r^{-1}\int_{ B^+(y_0,2r)\backslash B^+(y_0,r)} |\nabla \solu|^{p-1} |\tilde{\solu}|. 
\end{equation}
This can be further estimated in two ways. For the estimate \eqref{eq:growth:peqn}, by Young and Poincar\'e inequalities, we have for any $\mu >0$
\[
 \int_{{\R^n_+}}|\nabla \solu|^{p-2} (\nabla \eta_{B(y_0,r)}\tilde{\solu})\cdot \nabla \solu\  \aleq \frac{1}{\mu} \int_{ B^+(y_0,2r)\backslash B^+(y_0,r)} |\nabla \solu|^{p}  + \mu^{p-1} \int_{ B^+(y_0,2r)}|\nabla u|^{p}.
\]
For the estimates \eqref{eq:growth:pllnbd0} and \eqref{eq:growth:pllnbd}, by Young's inequality we have for any $\lambda > 0$ 
\[
 \int_{{\R^n_+}}|\nabla \solu|^{p-2} \nabla \eta_{B(y_0,r)}\cdot \nabla \solu\ \tilde{\solu} \aleq \lambda \int_{B^+(y_0,2r)} |\nabla \solu|^{p}+\lambda^{1-p} r^{-p}\int_{B^+(y_0,2r)} |\solu-(\solu)_{B^{{+}}(y_0,2r)}|^p.
\]
For the last term of \eqref{eq:gradientutildauwidehat}
\[
\begin{split}
 \int_{{\R^n_+}}|\nabla u|^{p-2}\nabla\eta_{B(y_0,r)}\, (\solu-(\solu)_{B^{{+}}(y_0,2r)})\cdot \nabla \tilde{\solu} &\aleq r^{-2}\int_{B^+(y_0,2r)\setminus B^+(y_0,r)} |\nabla u|^{p-2}|u-(u)_{B^{+}(y_0,2r)}|^2\\
 & + r^{-1}\int_{B^+(y_0,2r)\setminus B^+(y_0,r)} |\nabla u|^{p-1}|u-(u)_{B^{{+}}(y_0,2r)}|.
\end{split}
 \]
By a similar estimate, we easily get for any $\mu>0$
\[
 \int_{\R^n_+}|\nabla u|^{p-2}\nabla\eta_{B(y_0,2r)}\, (\solu-(\solu)_{B^{{+}}(y_0,2r)})\cdot \nabla \tilde{\solu} \aleq
 \frac{1}{\mu} \int_{ B^+(y_0,2r)\backslash B^+(y_0,r)} |\nabla \solu|^{p}  + \mu^{p-1} \int_{ B^+(y_0,2r)}|\nabla u|^{p}
 \]
and for any $\lambda>0$
\[
\begin{split}
\int_{\R^n_+}|\nabla u|^{p-2}\nabla\eta_{B(y_0,2r)}\, (\solu-(\solu)_{B^{{+}}(y_0,2r)})\cdot \nabla \tilde{\solu} &\aleq
 \lambda \int_{B^+(y_0,2r)} |\nabla \solu|^{p}\\
 &\quad +\lambda^{1-p} r^{-p}\int_{B^+(y_0,2r)} |\solu-(\solu)_{B^{{+}}(y_0,2r)}|^p.
 \end{split}
\]

Consequently, we found 
\begin{equation}\label{eq:intest}
\begin{split}
\int_{ B^+(y_0,r)}|\nabla \solu|^p &\aleq \left|\int_{\R^n_+} |\nabla \solu|^{p-2} \nabla \solu \cdot \nabla \tilde{\solu}\right| + \left|\int_{\R^n_+} |\nabla \solu|^{p-2} \nabla \solu \cdot \nabla \hat{\solu}\right|\\
& \quad +     \frac{1}{\mu} \int_{ B^+(y_0,2r)\backslash B^+(y_0,r)} |\nabla \solu|^{p}  + \mu^{p-1} \int_{ B^+(y_0,2r)}|\nabla u|^{p}
\end{split}
\end{equation}
and
\begin{equation}\label{eq:intest2}
\begin{split}
\int_{ B^+(y_0,r)}|\nabla \solu|^p &\aleq \left|\int_{\R^n_+} |\nabla \solu|^{p-2} \nabla \solu \cdot \nabla \tilde{\solu}\right| + \left|\int_{\R^n_+} |\nabla \solu|^{p-2} \nabla \solu \cdot \nabla \hat{\solu}\right|\\
& \quad + \lambda \int_{B^+(y_0,2r)} |\nabla \solu|^{p}+\lambda^{1-p} r^{-p}\int_{B^+(y_0,2r)} |\solu-(\solu)_{B^{{+}}(y_0,2r)}|^p.
\end{split}
\end{equation}

{If we are in the interior case, i.e.,} $B(y_0,2r) \subset {B^+(0,R)}$, then $\supp \tilde{u} \cup \supp \hat{u} \subset {B^+(0,R)}$ and thus $\div(|\nabla u|^{p-2} \nabla u) = 0$ in ${B^+(0,R)}$ implies
\[
 \left|\int_{\R^n_+} |\nabla \solu|^{p-2} \nabla \solu \cdot \nabla \tilde{\solu}\right| + \left|\int_{\R^n_+} |\nabla \solu|^{p-2} \nabla \solu \cdot \nabla \hat{\solu}\right| = 0.
\]
Thus, for $B(y_0,2r) \subset {B^+(0,R)}$ the claim is proven.

From now on we assume that {the ball $B(y_0,r)$ is close to the boundary, i.e, } $B(y_0,2r) {\cap \, \{\R^{n-1}\times\{0\}\}\neq \emptyset}$. By Lemma~\ref{la:goodpde},
\[
 \int_{\R^n_+} |\nabla \solu|^{p-2} \nabla \solu^i\cdot \nabla \tilde{\solu}^i = I + II + \frac{1}{2} III,
\]
where
\[
\begin{split} 
 I:=& \int_{\R^n_+} |\nabla \solu|^{p-2} \nabla \solu^k \cdot \Omega_{ik}\, \tilde{\solu}^i,\\
 II :=&\int_{\R^n_+} |\nabla \solu|^{p-2} \nabla \solu^i\cdot \nabla \brac{|\solu|^2-1} \ \tilde{\solu}^i,\\
 III:=&\int_{\R^n_+} |\nabla \solu|^{p-2} \nabla \tilde{\solu}^i \cdot \nabla \brac{|\solu|^2-1}\ \solu^i.\\
\end{split}
\]
Since $\solu$ is $p$-harmonic and by Lemma~\ref{la:conservation} all three terms above contain products of divergence-free and rotation-free quantities. However, the div-curl estimate by Coifman, Lions, Meyer, Semmes \cite{CLMS} is only applicable when at least one term vanishes at the boundary, otherwise there are counterexamples, see \cite{DaLio-Palmurella-2017,Hirsch-2019}.

We investigate the \underline{first term $I$}. Let $\tilde{B} \subset {B^+(0,R)}$ be a smooth, bounded, open, and convex set, such that ${B^+(y_0,2r)}\subset\tilde{B} \subset B(y_0,3r)$ and $\partial \tilde{B} \cap \partial \R^n_+ = B(y_0,2r) \cap \partial \R^n_+$. By Hodge decomposition\footnote{
{More precisely, one argues, e.g., as in \cite[(3.6), (3.7)]{Schikorra10}: One solves
\[
\begin{cases}
\lap \zeta_{ik}   = \curl(|\nabla \solu|^{p-2} \Omega_{ik}) &\mbox{in $\tilde{B}$}\\
\zeta_{ik} = 0 \quad& \mbox{on $\partial\tilde{B}$}.
\end{cases}
\]
such that \eqref{eq:hodge1st:2} is satisfied. 
Then one sets
\[
H := |\nabla \solu|^{p-2} \Omega_{ik} - {\rm Curl} \zeta_{ik}.
\]
By  Poincar\'{e} lemma we can write $H = \nabla \xi$.}

} (see \cite[(\textcolor{red}{10.4})]{Iwaniecmartin})  we find $\xi_{ik} \in W^{1,p'}(\tilde{B})$, with $p'=\frac{p}{p-1}$, and $\zeta_{ik} \in W^{1,p'}_0(\tilde{B},\Ep^{2}\R^n)$ such that 
\begin{equation}\label{eq:hodge1st}
 |\nabla \solu|^{p-2} \Omega_{ik} = \nabla \xi_{ik} +  \Curl\zeta_{ik} \quad \mbox{in $\tilde{B}$}.
\end{equation}
Moreover, we have
\begin{equation}\label{eq:hodge1st:2}
 \|\zeta_{ik}\|_{W^{1,p'}(\tilde{B})} \aleq \||\nabla \solu|^{p-2} \Omega_{ik}\|_{L^{p'}(B(y_0,3r))}.
\end{equation}
The boundary data of $\zeta$ and Lemma~\ref{la:conservation} imply that 
\[
 \int_{\tilde{B}} \nabla \xi_{ik} \cdot \nabla \varphi  = {\int_{\tilde{B}}  |\nabla \solu|^{p-2} \Omega_{ik} \cdot \nabla \varphi- \int_{\tilde{B}}  \Curl \zeta_{ik}\cdot \nabla \varphi } = 0 \quad \mbox{for any $\varphi \in C^\infty(\overline{\tilde{B}})$}.
\]
That is, $\xi_{ik}$ is harmonic with trivial Neumann data, and thus $\xi_{ik}$ is constant. In particular, \eqref{eq:hodge1st} simplifies to
\begin{equation}\label{eq:hodge2nd}
 |\nabla \solu|^{p-2} \Omega_{ik} = \Curl \zeta_{ik} \quad \mbox{in $\tilde{B}$}.
\end{equation}
Consequently,
\[
 I =  \int_{\R^n_+} \Curl \zeta_{ik} \cdot \nabla \solu^k\, \tilde{\solu}^i = \int_{\R^n} \Curl \zeta_{ik} \cdot \nabla \solu^k\, \tilde{\solu}^i. 
\]
The last equality is true, since $\zeta_{ik}$ {vanishes on} ${\partial \R^n_+\cap B(0,R)}$ and we can extend it by zero to ${\R^n_-\cap B(0,R)}$. Now we use the div-curl structure and apply the result by Coifman, Lions, Meyer, Semmes  \cite{CLMS}. Recall that $BMO$ is the space of functions $f$ with finite {seminorm} $[f]_{BMO} < \infty$. Here,
\[
 [f]_{BMO} := \sup_{B} |B|^{-1} \int_{B} |f-(f)_{B}|,
\]
where the supremum is taken over all balls $B$. Observe that by Poincar\'e inequality,
\begin{equation}\label{eq:BMOpoincare}
 [f]_{BMO} \aleq \sup_{x_0 \in \R^n, \ \rho>0} \brac{\rho^{p-n} \int_{B(x_0,\rho)} |\nabla f|^p}^{\frac{1}{p}}.
\end{equation}
Coifman, Lions, Meyer, Semmes  showed in \cite{CLMS} that the following inequality holds
\[
 \int_{\R^n} F \cdot G\ \varphi \aleq \|F\|_{L^p(\R^n)}\, \|G\|_{L^{p'}(\R^n)}\, [\varphi]_{BMO}
\]
whenever $F$ and $G$ are vector fields such that $\div F = 0$ and $\curl G = 0$. See also \cite{Lenzmann-Schikorra-2016} for a different proof. In our situation this inequality implies\footnote{{Here, $\tilde{u}$ is extended into the whole space $\R^n$ in such a way that $[\tilde{u}]_{BMO}\aleq \lambda$. This can be done by an appropriate reflection of $u$ outside of $B^+(y_0,3r)$.}}
{\begin{equation}\label{eq:utildainbmo}
\begin{split}
  |I| &\aleq \||\nabla \solu|^{p-2} \Omega_{ik}\|_{L^{p'}(B^+(y_0,4r))}\ \|\nabla \solu\|_{L^p(B^+(y_0,4r))}\ [\tilde{\solu}]_{BMO}\\
 &\aleq \|\nabla \solu\|_{L^p(B^+(y_0,4r))}^p\ [\tilde{\solu}]_{BMO}.
\end{split}
 \end{equation}}
The last estimate follows readily from the definition of $\Omega$ in Lemma~\ref{la:conservation}.
Thus, for the $\lambda$ from \eqref{small} we obtain
\[
 |I| \aleq \lambda\ \int_{B^+(y_0,4r)} |\nabla \solu|^p.
\]
\underline{As for $II$}, since $\div(|\nabla \solu|^{p-2} \nabla \solu)=0$ in ${B^+(0,R)}$, there exists $\zeta_i\in W^{1,p}(B^+(y_0,2r),\Ep^2 \R^n)$ such that 
\[
 |\nabla \solu|^{p-2} \nabla \solu^i = \Curl \zeta_i \quad \mbox{in $B^+(y_0,2r)$}.
\]
We can extend $\zeta$ to all of $\R^n$ so that 
\[
 \|\zeta\|_{W^{1,p'}(\R^n)} \aleq \|\nabla \solu\|_{L^p(B^{+}(y_0,2r))}^{p-1}.
\]
{Also, since $u$ is assumed to be bounded we have $|u|^2 \in {W}^{1,p}(B^+(0,R))$, and in the sense of traces $|u|^2 \equiv 1$ on $B(0,R)\cap\{\R^{n-1} \times \{0\}\}$. This is equivalent to saying that the extension of $|u|^2-1$ by zero to $B(0,R)\cap \,\R^n_-$ belongs to ${W}^{1,p}({B(0,R)})$, that is we have, $(|u|^2-1)\chi_{\R^n_+} \in W^{1,p}({B(0,R)})$ and the distributional gradient satisfies}
{
\[
 \nabla \left ((|u|^2-1)\chi_{\R^n_+} \right) = \chi_{\R^n_+}\nabla |u|^2 \quad \text{a.e. in } B(0,R).
\]
}
In particular, since $(|u|^2-1)\chi_{\R^n_+}$ is zero on $B(y_0,2r)\cap\, \R^n_{-}$ we can use Poincar\'e inequality to get
\begin{equation}\label{eq:poinczero}
 \||\solu|^2 -1\|_{L^p( B^+(y_0,2r))} \aleq r\, \|u\|_{L^{\infty}(B^+(y_0,4r))}\, \|\nabla \solu\|_{L^p(B^+(y_0,4r))}.
\end{equation}
In particular, by using that $|\nabla \eta_{B(y_0,2r)}|\aleq r^{-1}$, \eqref{eq:BMOpoincare}, the triangle inequality in $L^p$ and \eqref{eq:poinczero}, for the $\lambda$ from \eqref{small},
\[
 [\brac{|\solu|^2-1}\chi_{\R^n_+} \eta_{B(y_0,{2}r)}]_{BMO} \aleq \lambda.
\]
We also observe that $\nabla \tilde{u} \equiv \eta_{B(y_0,2r)} \nabla \tilde{u}$. 
Thus, integrating by parts we obtain
\[
 II =-\int_{\R^n} \Curl \zeta\cdot \nabla \tilde{\solu}^i\ \brac{|\solu|^2-1}\chi_{\R^n_+}\eta_{B(y_0,{2}r)}.
\]
Hence, with the div-curl theorem from \cite{CLMS}, see also the localized version \cite[Corollary 3]{Strzelecki-1994}, we find
\[
 |II| \aleq \lambda \|\nabla \solu\|_{L^p(B^{+}(y_0,4r))}^{p}. 
\]
It remains \underline{to treat $III$}. Observe that 
\[
\begin{split}
 \nabla \tilde{\solu}^i &\cdot \nabla \brac{|\solu|^2-1}\ \solu^i \\
 &=\nabla \solu^i \cdot \nabla \brac{|\solu|^2-1}\ \eta_{B(y_0,r)} \solu^i + \nabla \eta_{B(y_0,r)}\, (\solu^i-(\solu^i)_{B^{{+}}(y_0,2r)}) \cdot \nabla \brac{|\solu|^2-1}\ \solu^i\\
  &=\nabla \solu^i \cdot \nabla \brac{|\solu|^2-1}\ \tilde{\solu}^i \\
  &\quad+\nabla \solu^i \cdot \nabla \brac{|\solu|^2-1}\ \eta_{B(y_0,r)} (\solu^i)_{B^{{+}}(y_0,2r)} \\
  &\quad+ \nabla \eta_{B(y_0,r)}\, (\solu^i-(\solu^i)_{B^{{+}}(y_0,2r)}) \cdot \nabla \brac{|\solu|^2-1}\ \solu^i.
\end{split}
\]
By integration by parts, using that $\div(|\nabla u|^{p-2} \nabla u) = 0$ {in $B^+(0,R)$}, $|u|^2-1$ is zero on {$\partial \R^n_+\cap B(0,R)$} and then arguing as in the argument for $II$,
\[
\begin{split}
 \int_{\R^n_+} |\nabla \solu|^{p-2} \nabla \solu^i \cdot \nabla \brac{|\solu|^2-1}\ \tilde{\solu}^i
 &= -\int_{\R^n_+} |\nabla \solu|^{p-2} \nabla \solu^i \cdot \nabla \tilde{\solu}^i\ \brac{|\solu|^2-1}\\
 &\aleq  \lambda\, \|\nabla \solu\|_{L^p(B^{+}(y_0,4r))}^{p}.
\end{split}
 \]
Moreover, again since $\div(|\nabla \solu|^{p-2} \nabla \solu) = 0$ in ${B^+(0,R)}$ and $|\solu|^2 -1$ is zero on $\partial \R^n_+\cap B(0,R)$,
\[
\begin{split}
 \Bigg| \int_{\R^n_+}& |\nabla \solu|^{p-2}\nabla \solu^i \cdot \nabla \brac{|\solu|^2-1}\ \eta_{B(y_0,r)} (\solu^i)_{B^{{+}}(y_0,2r)} \Bigg|\\
  &=\left|\int_{\R^n_+} |\nabla \solu|^{p-2}\nabla \solu^i \cdot \brac{|\solu|^2-1}\ \nabla \eta_{B(y_0,r)} (\solu^i)_{B^{{+}}(y_0,2r)}\right|\\
  &\aleq r^{-1}\, \|u\|_{L^\infty({B^+(0,R)})}\, \|\nabla \solu\|^{p-1}_{L^p( B^+(y_0,2r)\backslash B^+(y_0,r))}\ \||\solu|^2 -1\|_{L^p( B^+(y_0,2r))}.
\end{split}
  \]
This leads to two estimates. Firstly, if we want to find \eqref{eq:growth:pllnbd}, by Young's inequality,
\[
\begin{split}
  \int_{\R^n_+}& |\nabla \solu|^{p-2}\nabla \solu^i \cdot \nabla \brac{|\solu|^2-1}\ \eta_{B(y_0,r)} (\solu^i)_{B^{{+}}(y_0,r)}\\
  &\aleq \lambda \|\nabla \solu\|^{p}_{L^p( B^+(y_0,2r))} + \lambda^{1-p}\, r^{-p} \||\solu|^2 -1\|^p_{L^p( B^+(y_0,2r))}.
\end{split}
 \]
Secondly, for \eqref{eq:growth:peqn} by \eqref{eq:poinczero} and 
by Young's inequality we have for any $\mu > 0$
\[
\begin{split}
  \int_{\R^n_+}& |\nabla \solu|^{p-2}\nabla \solu^i \cdot \nabla \brac{|\solu|^2-1}\ \eta_{B(y_0,r)} (\solu^i)_{B^{{+}}(y_0,2r)}\\
  &\aleq \mu^{-1} \|\nabla \solu\|^{p}_{L^p( B^+(y_0,2r) \backslash B^+(y_0,r))} + \mu^{p-1}\|\nabla \solu\|^p_{L^p( B^+(y_0,2r))}.
\end{split}
 \]
The last remaining term can be treated in a similar way and we have
\[
\begin{split}
   \int_{\R^n_+}& |\nabla \solu|^{p-2}\nabla \eta_{B(y_0,r)}\, (\solu^i-(\solu^i)_{B^{{+}}(y_0,2r)}) \cdot \nabla \brac{|\solu|^2-1}\ \solu^i \\
  &\aleq \mu^{-1} \|\nabla \solu\|^{p}_{L^p( B^+(y_0,2r)\backslash B^+(y_0,r))} + \mu^{p-1}\|\nabla \solu\|^p_{L^p( B^+(y_0,2r))}
\end{split}
 \]
and
\[
\begin{split}
   \int_{\R^n_+}& |\nabla \solu|^{p-2}\nabla \eta_{B(y_0,r)}\, (\solu^i-(\solu^i)_{B^{{+}}(y_0,2r)}) \cdot \nabla \brac{|\solu|^2-1}\ \solu^i \\
  &\aleq \lambda \|\nabla \solu\|^{p}_{L^p( B^+(y_0,2r))} + \lambda^{1-p}\, r^{-p} \|u- (u)_{B^{{+}}(y_0,2r)}\|^p_{L^p( B^+(y_0,2r))}.
\end{split}
 \]
Combining the estimates of $I$, $II$, and $III$ and plugging them into estimates \eqref{eq:intest} and \eqref{eq:intest2}, we conclude.
\end{proof}

\section{H\"older regularity for the case \texorpdfstring{$p=n$}{p=n}}\label{s:pnHoelder}
For the case $p=n$ H\"older continuity of the solution $u$ from Theorem~\ref{th:main} follows from Proposition~\ref{pr:growth} by a standard iteration argument. For higher regularity, and for $p < n$, we need to combine the growth estimate{s} from Proposition~\ref{pr:growth} with the reflection method.

\begin{proposition}[$\epsilon$-regularity for $p=n$: H\"older continuity]\label{pr:main:hoelderpn} 
Let $D \subset \R^n$ be a smooth, bounded domain, then there are positive constants $\epsilon = \epsilon(n,D)$, $\alpha = \alpha(n,D)$ such that the following holds for $p=n$:

Any solution $u\in W^{1,n}(D,\R^N)$ to \eqref{eq:soluspheredef} that satisfies for an $R > 0$ and for an $x_0 \in \overline{D}$
\[
\int_{B(x_0,R)\cap D} |\nabla \solu|^n < \epsilon
\]
is H\"older continuous in $B(x_0,R/2) \cap \overline{D}$. Moreover, we have the estimate
\[
 \sup_{x, y \in B(x_0,R/2)\cap \overline{D}}  \frac{|u(x)-u(y)|}{|x-y|^{\alpha}} \aleq R^{-\alpha} \|\nabla u\|_{L^n(B(x_0,R)\cap D)}.
\]
\end{proposition}
\begin{proof}
Let $\lambda := \epsilon^{\frac{1}{n}}$ and apply Proposition~\ref{pr:growth} to any $B(y_0,4r)\subset B(x_0,R/2)$, for $\mu > 0$ to be chosen below. We add 
\[
 C\mu^{-1} \int_{B(y_0,r)\cap D} |\nabla u|^n
\]
to both sides of \eqref{eq:growth:peqn}. Then we find
\[
\brac{1+C\mu^{-1}} \int_{B(y_0,r)\cap D} |\nabla u|^n \leq C\ \brac{\epsilon^{\frac{1}{n}} + \mu^{n-1}+\mu^{-1}} \int_{B(y_0,4r)\cap D} |\nabla u|^n.
\]
We choose $\epsilon, \mu >0$ small enough so that $\tau < 1$, where
\[
 \tau := \brac{\frac{C\ \brac{\epsilon^{\frac{1}{n}} + \mu^{n-1}+\mu^{-1}}}{1+C\mu^{-1}}}^{\frac{1}{n}}.
\]
We have for any $B(y_0,4r)\subset B(x_0,R/2)$ 
\[
\|\nabla u\|_{L^n(B(y_0,r)\cap D)} \leq \tau \|\nabla u\|_{L^n(B(y_0,4r)\cap D)}.
\]
Iterating this on successively smaller balls, cf. e.g. \cite[Chapter III, Lemma 2.1]{GiaquintaMultipleIntegrals}, we find that for a uniform $\alpha = \alpha(\tau) > 0$ and for any $B(y_0,4r) \subset B(x_0,R/2)$,
\[
 \|\nabla u\|_{L^n(B(y_0,r)\cap D)} \aleq \brac{\frac{r}{R}}^{\alpha} \|\nabla u\|_{L^n(B(x_0,R)\cap D)}.
\]
In particular, we have by Poincar\'e inequality 
\[
 \sup_{B(y_0,4r) \subset B(x_0,R/2)} r^{-\alpha-1}\|u-(u)_{B(y_0,r){\cap D}}\|_{L^n(B(y_0,r)\cap D)} \aleq R^{-\alpha}\, \|\nabla u\|_{L^n(B(x_0,R)\cap D)}.
\]
By the characterization of Campanato spaces and H\"older spaces, e.g. see \cite[Chapter III, p.75]{GiaquintaMultipleIntegrals}, this implies
\[
 \sup_{x, y \in B(x_0,R/2){\cap} \overline{D}}  \frac{|u(x)-u(y)|}{|x-y|^{\alpha}} \aleq R^{-\alpha} \|\nabla u\|_{L^n(B(x_0,R)\cap D)}.
\]
\end{proof}

\section{H\"older-continuity for solutions to a supercritical system}\label{s:genericHoelder}
In Proposition~\ref{pr:growth} we showed that solutions from Theorem~\ref{th:main} satisfy certain growth estimates. For $p=n$ these growth estimates imply H\"older continuity by an iteration argument, as we have seen in Proposition~\ref{pr:main:hoelderpn}. 

For $p < n$ more work is needed. The following Proposition shows that under a smallness assumption solutions to systems satisfying
\begin{equation}\label{eq:plapnablaup} |\div(|\nabla u|^{p-2} \nabla u)| \aleq |\nabla u|^p \end{equation} 
are H\"older continuous once the growth conditions from Proposition~\ref{pr:growth} are satisfied, that is when \eqref{eq:growthcond} and \eqref{eq:growthcondinter} below are assumed \emph{a priori}. 
Observe that without assuming \emph{a priori} the growth conditions \eqref{eq:growthcond} and \eqref{eq:growthcondinter} below on the solution $u$, there is no hope for proving \emph{any} regularity for solutions to a systems that have a structure of \eqref{eq:plapnablaup}. Indeed, it is easy to check that $\log \log \frac{2}{|x|}$ and $\sin \log \log \frac{2}{|x|}$ satisfy \eqref{eq:plapnablaup} for $p= n$.

In the next section, in order to prove Theorem~\ref{th:main}, we use the reflection method from Scheven's \cite{Scheven-2006} to obtain an equation of the form \eqref{eq:gensys:eq}. Since we already obtained the necessary growth estimates in Proposition~\ref{pr:growth}, the following proposition then leads to regularity.
\begin{proposition}\label{pr:gensysreg}
Let $D \subset \R^n$ be a {smooth, bounded domain} and let $\mathcal{M}$ be a smooth, compact $(n-1)$-dimensional manifold.
%
%
Assume that $u \in W^{1,p}(D,\R^N)$ is a solution to
\begin{equation}\label{eq:gensys:eq}
 \div(|G(x)\nabla u(x)|^{p-2} G(x)\nabla u(x)) = f_{{u}}(x),
\end{equation}
where $f_{{u}} \in L^1(D,\R^N)$ satisfies the following  estimate
\begin{equation}\label{eq:gensys:fgrowth}
 |f_{{u}}(x)| \leq C\ |\nabla u(x)|^p
\end{equation}
and $G \in C^\infty(\overline{D},GL(n))$. 

Moreover, assume \emph{a priori} that for every $B(x_0,R) \subset D$, $\lambda > 0$ such that
\begin{equation}\label{eq:gensys:lambda}
 \sup_{B(y_0,r)\subset  B(x_0,R)} r^{p-n} \int_{B(y_0,r)} |\nabla \solu|^p < \lambda^p
\end{equation}
the solution $u$ already satisfies the following growth condition on any $B(y_0,4r) \subset B(x_0,R)$:

If $B(y_0,2r) \cap \mathcal{M} = \emptyset$, then
\begin{equation}\label{eq:growthcond}
 \int_{B(y_0,r)} |\nabla u|^p \leq C \lambda \int_{B(y_0,4r)} |\nabla u|^p + C \lambda^{1-p} r^{-p}\int_{B(y_0,4r)} |\solu-(\solu)_{B(y_0,4r)}|^p
\end{equation}
and, if $B(y_0,2r) \cap \mathcal{M} \neq \emptyset$, then
\begin{equation}\label{eq:growthcondinter}
\begin{split}
 \int_{B(y_0,r)} |\nabla u|^p \leq C \lambda \int_{B(y_0,4r)} |\nabla u|^p &+ C \lambda^{1-p} r^{-p}\int_{B(y_0,4r)} |\solu-(\solu)_{B(y_0,4r)}|^p\\
 &+ C\lambda^{1-p}r^{-p}\int_{B(y_0,4r)} |\solu- (\solu)_{B(y_0,4r)\cap \mathcal{M}}|^p\\
 &+C\lambda^{1-p}r^{1-p}\int_{B(y_0,4r)\cap \mathcal{M}} |\solu- (\solu)_{B(y_0,4r)\cap \mathcal{M}}|^p.
\end{split}
\end{equation}
Then there exist constants $\alpha = \alpha(G,p,n,C,D)$, $\epsilon>0$ such that if \eqref{eq:gensys:lambda} holds on some $B(x_0,R) \subset D$ for $\lambda < \epsilon$, then $u \in C^{\alpha}(B(x_0,R/2),{\R^N})$. Moreover, we have the estimate
\[
 \sup_{x, y \in B(x_0,R/2)}  \frac{|u(x)-u(y)|}{|x-y|^{\alpha}} \leq C_0\, R^{-\alpha} \brac{\sup_{B(y_0,r)\subset  B(x_0,R)} r^{p-n} \int_{B(y_0,r)} |\nabla \solu|^p}^{\frac{1}{p}}.
\]
The constant $C_0$ depends on $\mathcal{M}$, $D$, $C$, and $G$.
\end{proposition}

To prove Proposition~\ref{pr:gensysreg} we follow the strategy developed in \cite{Hardt-Kinderlehrer-Lin-1986} and \cite[Theorem 2.4]{Hardt-Lin-1987}. The crucial result is that the equation for $u$ together with the growth assumptions \eqref{eq:growthcond} and \eqref{eq:growthcondinter} on $u$ imply the following decay estimate.

\begin{proposition}\label{pr:decay}
There are uniform constants $\epsilon, \theta \in (0,1)$ and $\overline{R} = \overline{R}(\mathcal{M}) \in (0,1)$ so that the following holds:

Let $u$ and {$D$} be as in Proposition~\ref{pr:gensysreg} and assume that for a ball $B(x_0,R) \subset {D}$ and $R \in (0,\overline{R})$ it holds
\begin{equation}\label{eq:small2}
 \lE{x_0,R}(\solu) := \sup_{B(y_0,r)\subset  B(x_0,R)} r^{p-n} \int_{B(y_0,r)} |\nabla \solu|^p < \epsilon^p.
\end{equation}
Then
\begin{equation}\label{eq:Ex0tleqhe}
\lE{x_0,\theta R}(\solu) \leq \frac{1}{2}\lE{x_0,R}(\solu).
\end{equation}
\end{proposition}

\begin{proof}
 It suffices to prove
 \begin{equation}\label{eq:insteadEx0tleqhe}
(\theta R)^{p-n} \int_{B(y_0,\theta R)} |\nabla \solu|^p \leq \frac{1}{2}\lE{x_0,R}(\solu) \quad \text{for any } B(y_0,4\theta R) \subset B(x_0,R/2).
\end{equation}
Indeed, \eqref{eq:Ex0tleqhe}  follows from \eqref{eq:insteadEx0tleqhe} by taking smaller $\theta$ and observing that $B(x_1,R_1) \subset B(x_2,R_2)$ implies $\lE{x_1,R_1}(\solu) \leq \lE{x_2,R_2}(\solu)$.

%

Assume the claim \eqref{eq:insteadEx0tleqhe} is false. Then, for any $\theta \in (0,1)$ we have a sequence of balls with $B(y_i,4\theta R_i)\subset B(x_i,R_i/2)\subset {D}$, a sequence $(\epsilon_i)_{i=1}^\infty$ satisfying $\lim_{i \to \infty} \epsilon_i = 0$, and a sequence $(\solu_i)_{i=1}^\infty \subset W^{1,p}({D},\R^N)$ of solutions to \eqref{eq:gensys:eq} satisfying the growth assumptions of Proposition~\ref{pr:gensysreg}, such that
\begin{equation}\label{eq:small2i}
 \sup_{B(y,r)\subset  B(x_i,R_i)} r^{p-n} \int_{B(y,r)} |\nabla \solu_i|^p = \epsilon_i^p,
\end{equation}
but
\begin{equation}\label{eq:notsmall}
(\theta R_i)^{p-n} \int_{B(y_i,\theta R_i)} |\nabla \solu_i|^p > \frac{1}{2} \epsilon_i^p.
\end{equation}
For simplicity, we assume that $R_i \equiv R_0$ and $x_i \equiv x_0$ for some $R_0 > 0$ and $x_0 \in \R^n$.

{
This is no loss of generality, since we can rescale the maps $u$ by the factor $R_0/R_i$. Observe that this rescales the manifold $\mathcal{M}$, but in a way that \eqref{eq:growthcondinter} still holds.}
%
%
Set 
\[
 w_i := \frac{1}{\epsilon_i}(u_i-(u_i)_{B(x_0,R_0)}).
\]
Clearly,
\begin{equation*}\label{eq:wimeanvalue}
 (w_i)_{B(x_0,R_0)} = 0\quad \mbox{for all $i \in \N$}.
\end{equation*}
Thus, we can apply Poincar\'e inequality and have by \eqref{eq:small2i},
\[
  \sup_{i \in \N}\|\nabla w_i \|^p_{L^p(B(x_0,R_0))} \aleq R_0^{n-p} \quad \mbox{and} \quad \sup_{i \in \N} \|w_i\|_{L^p(B(x_0,R_0))}^p \aleq R_0^{n-p+1}. 
\]
Thus, up to a subsequence denoted again by $w_i$, we find $w \in W^{1,p}(B(x_0,R_0),\R^N)$ such that as $i \to \infty$,
\begin{align*}
 w_i &\rightharpoonup w & &\mbox{weakly in } W^{1,p}(B(x_0,R_0)),\\
 w_i& \to w & &\mbox{strongly in } L^p(B(x_0,R_0)),\\
 w_i& \to w & &\mbox{strongly in } L^p(B(x_0,R_0) \cap \mathcal{M},d\H^{n-1}), \\
 w_i& \to w  & &\H^n \mbox{-a.e. on $B(x_0,R_0)$ and $\H^{n-1}$ \mbox{-a.e. on } $B(x_0,R_0) \cap \mathcal{M}$}.
\end{align*}
In particular,
\begin{equation}\label{eq:wmeanvalue}
 (w)_{B(x_0,R_0)} = 0,
\end{equation}
also
\[
 \|\nabla w \|^p_{L^p(B(x_0,R_0))} \aleq R_0^{n-p} \quad \mbox{and} \quad \|w\|^p_{L^p(B(x_0,R_0))} \aleq R_0^{n-p+1}. 
\]

Moreover, for any $\varphi \in C_c^\infty(B(x_0,R_0))$,
\[
 \int_{B(x_0,R_0)} |G\nabla \solw_i|^{p-2}G\nabla \solw_i\cdot \nabla \varphi = (\epsilon_i)^{1-p} \int_{B(x_0,R_0)} |G\nabla u_i|^{p-2}G\nabla u_i\cdot \nabla \varphi.
\]
Now, by \eqref{eq:gensys:eq} and \eqref{eq:gensys:fgrowth},
\[
 \left |\int_{B(x_0,R_0)} |G\nabla \solw_i|^{p-2}G\nabla \solw_i\cdot \nabla \varphi \right |\aleq (\epsilon_i)^{1-p} \|\varphi\|_{L^\infty(B(x_0,R_0))}\ \|\nabla u_i\|^p_{L^p(B(x_0,R_0))}.
\]
That is, by \eqref{eq:small2i}
\[
 \left |\int_{B(x_0,R_0)} |G\nabla \solw_i|^{p-2}G\nabla \solw_i\cdot \nabla \varphi \right | \aleq {\|\varphi\|_{L^\infty(B(x_0,R_0))}}R_0^{n-p} \epsilon_i\le \epsilon_i{\|\varphi\|_{L^\infty(B(x_0,R_0))}}.
\]

Now as in \cite[Section~4]{Dolzmann-Hungerbuehler-Mueller-1997}
\begin{equation}\label{eq:plapsolw0}
 \div(|G\nabla \solw|^{p-2} G\nabla \solw) = 0 \quad \mbox{in $B(x_0,R_0)$}.
\end{equation}

From \eqref{eq:wmeanvalue} and the Lipschitz estimates for solutions to \eqref{eq:plapsolw0}, see \cite{Uhlenbeck-1977} as well as \cite{Mingione-2011,Duzaar-Mingione-2011} and in particular \cite[(1.7)]{Kuusi-Mingione-2012}, we have for any $B(z,r) \subset B(x_0,R_0/2)$,
\begin{equation*}\label{eq:Lipschitzestplaplace}
 r^{-n}\int_{B(z,r)} |\solw - (\solw)_{B(z,r)}|^p \aleq r^p,
\end{equation*}
and if additionally  $B(z,r) \cap \mathcal{M} \neq \emptyset$ and $r<\overline{R}$ for $\overline{R} = \overline{R}(\mathcal{M})$ small enough, then
\[
 r^{1-n}\int_{\mathcal{M} \cap B(z,r)} |\solw - (\solw)_{\mathcal{M} \cap B(z,r)}|^p  + r^{-n}\int_{B(z,r)} |\solw - (\solw)_{\mathcal{M} \cap B(z,r)}|^p \aleq r^p.
\]
On the other hand, by strong $L^p$-convergence of $w_i$ to $w$, we find $i(\theta) \in \N$ so that for $i \geq i(\theta)$ and for any $r \in (\theta R_0,R_0)$ such that $B(z,r) \subset B(x_0,R_0)$, 
\begin{equation*}\label{eq:strongconvestimate}
 r^{1-n} \int_{B(z,r)\cap \mathcal{M}} |w_i-w|^p + r^{-n} \int_{B(z,r)} |w_i-w|^p \leq \theta^p.
\end{equation*}
Combining these estimates we get for any $i \geq i(\theta)$ and for any $r \in (\theta R_0,R_0)$ such that $B(z,r) \subset B(x_0,R_0/2)$,
\[
 r^{-n} \int_{B(z,r)} |u_i-(u_i)_{B(z,r)}|^p = \epsilon_i^p r^{-n} \int_{B(z,r)} |w_i-(w_i)_{B(z,r)}|^p \aleq \epsilon_i^p \brac{r^p + \theta^p}.
\]
If additionally $B(z,r) \cap \mathcal{M} \neq \emptyset$, then 
\[
r^{-n}\int_{B(z,r)} |u_i-(u_i)_{B(z,r)\cap \mathcal{M}}|^p = \epsilon_i^p\, r^{-n}\int_{B(z,r)} |w_i-(w_i)_{B(z,r)\cap \mathcal{M}}|^p \aleq \epsilon_i^p \brac{r^p + \theta^p}
\]
and
\[
r^{1-n}\int_{B(z,r)\cap \mathcal{M}} |u_i-(u_i)_{B(z,r)\cap \mathcal{M}}|^p \aleq \epsilon_i^p \brac{r^p + \theta^p}.
\]
{
We now apply the growth estimates \eqref{eq:growthcond} and \eqref{eq:growthcondinter} with $\lambda=\epsilon _0 \geq \epsilon_i$ of the solutions $u_i$ to find 
\[
(\theta R_0)^{p-n} \int_{B(y_i,\theta R_0)} |\nabla \solu_i|^p \leq C\ \epsilon_i^{p} \, \left (\epsilon_0+\epsilon _0^{1-p}\theta^p \right).
\]

By choosing $\epsilon_0$ and $\theta$ sufficiently small so that $\epsilon_0+\epsilon_0^{1-p}\theta^p<1/2$ we arrive at a contradiction with \eqref{eq:notsmall}.}
\end{proof}

\begin{proof}[Proof of Proposition~\ref{pr:gensysreg}]
We argue as in the proof of Proposition~\ref{pr:main:hoelderpn}: Assume that \eqref{eq:gensys:lambda} is satisfied on $B(x_0,R)$ for some $\lambda < \epsilon$.
Iterating the estimate from Proposition~\ref{pr:decay} on successively smaller balls, cf. \cite[Chapter III, Lemma 2.1]{GiaquintaMultipleIntegrals}, we find a small $\alpha > 0$ such that for all $r < R$ and $B(y_0,r) \subset B(x_0,R/2)$,
\[
 r^{p-n}\int_{B(y_0,r)} |\nabla u|^p \aleq \brac{\frac{r}{R}}^{\alpha p} E(x_0,R).
\]
In particular, for all $r < R$ and $B(y_0,r) \subset B(x_0,R/2)$,
\[
 r^{-\alpha p-n}\int_{B(y_0,r)} |u-(u)_{B(y_0,r)}|^p \aleq r^{p-\alpha p-n}\int_{B(y_0,r)} |\nabla u|^p \aleq R^{-\alpha p} E(x_0,R).
\]
We conclude by the identification of Campanato and H\"{o}lder spaces, see \cite[Chapter III, p.75]{GiaquintaMultipleIntegrals}. 
\end{proof}

\section{\texorpdfstring{$\epsilon$}{}-regularity: Proof of Theorem~\ref{th:main}}\label{s:proofthmain}
The proof of Theorem~\ref{th:main} is a combination of the growth estimate for solutions, Proposition~\ref{pr:growth}, the reflection method as in Scheven's \cite{Scheven-2006}, and Proposition~\ref{pr:gensysreg}. 
More precisely, we use the reflection method to find a solution to \eqref{eq:gensys:eq} from Proposition~\ref{pr:gensysreg}. The growth estimates \eqref{eq:growthcond} and \eqref{eq:growthcondinter} required in Proposition~\ref{pr:gensysreg} come from Proposition~\ref{pr:growth}: They hold for the unreflected solution and by an easy argument hold also for the reflection. 
To set up the reflection method we first gather some standard results. 
\begin{lemma}\label{la:distbd}
Let $D$ be a smooth, bounded domain in $\R^n$. There exists some $R_0 = R_0(D)$ such that the following holds for any $R \in (0,R_0)$. Let $\solu\in W^{1,p}(D,\R^N)$ be a solution to \eqref{eq:soluspheredef} and $\epsilon \in (0,1)$. If
\begin{equation}\label{eq:distbdsmallness}
 \sup_{B(y_0,r)\subset  B(x_0,R)} r^{p-n} \int_{B(y_0,r)\cap D} |\nabla \solu|^p < \epsilon^p
\end{equation}
and $B(x_0,R/2) \cap \partial D \neq \emptyset$, then
\[
 \sup_{x \in B(x_0,R/2) \cap D} \dist(\solu(x),\S^{N-1}) \leq C\epsilon.
\]
Here $C$ is a constant depending on $\partial D$.
\end{lemma}
\begin{proof}
Fix $x \in B(x_0,R/2){\cap D}$. Let $r := \frac{1}{10} \dist(x,\partial D)$, then by \eqref{eq:distbdsmallness} and the interior Lipschitz regularity {for} the $p$-Laplace equation, see \cite[(1.7)]{Kuusi-Mingione-2012},
\[
 |u(x)-(u)_{B(x,r)}|^p \aleq r^{p-n} \int_{B(x,5r)} | \nabla \solu|^p \leq \epsilon^p.
\]
Denote by $z_1 \in \partial D\cap B(x_0,R/2)$ the projection of $x$ onto $\partial D\cap B(x_0,R/2)$. Here we assume that $R < R_0$ for $R_0 = R_0(D)$ small enough {such that $z_1$ is well-defined.} 

{Let $y_0,y_1,\ldots,y_{10}$ be pairwise equidistant points on the line $[x,z_1]$ where $y_0 = x$ and $y_{10} = z_1$. That is, $|y_{i} - y_{i+1}| = r$.}

By triangle inequality, Poincar\'e inequality and again by \eqref{eq:distbdsmallness}, 
\[
\begin{split}
 |(u)_{B(x,r)}& - (u)_{B(z_1,r)\cap D}|^p\\ 
 &{\aleq \sum_{i=0}^{10} |(u)_{B(y_i,r)\cap D} -(u)_{B(y_{i+1},r) \cap D}|^p}\\
 &\aleq {\sum_{i=0}^{10} r^{p-n} \int_{B(y_i,4r) \cap D} |\nabla u|^p}\\
 &\aleq {\epsilon^p}. 
 \end{split}
\]
{From the second to third line, before applying Poincar\'e inequality, we also used that $|y_i-y_{i+1}| = r$, and thus (cf. footnote \ref{ft:AB})
\[
 |(u)_{B(y_i,r)\cap D} -(u)_{B(y_{i+1},r) \cap D}|^p \aleq \mvint_{B(y_i,4r) \cap D} |u -(u)_{B(y_i,4r) \cap D}|^p
\]}

Now for any $z_2 \in \partial D$
\[
 \dist((u)_{B(z_1,r)\cap D},\S^{N-1}) \aleq r^{-n} \int_{B(z_1,r) \cap D}|u(z_3)-u(z_2)|\ dz_3.
\]
Integrating $z_2$ over $\partial D \cap B(z_1,r)$ we find 
\[
\begin{split}
 \dist((u)_{B(z_1,r)\cap D},\S^{N-1}) &\aleq r^{-n} \int_{B(z_1,r) \cap D}|u(z_3)-(u)_{B(z_1,r) \cap \partial D}|\ dz_3\\
 &\quad+r^{1-n} \int_{B(z_1,r) \cap \partial D}|u(z_2)-(u)_{B(z_1,r) \cap \partial D}|\ dz_2.
\end{split}
 \]
By Poincar\'e inequality, trace theorem, and \eqref{eq:distbdsmallness}
\[
 \dist((u)_{B(z_1,r)\cap D},\S^{N-1}) \aleq \epsilon.
\]
Now the claim follows by triangle inequality for the distance,
\[
\begin{split}
 \dist(u(x),\S^{N-1}) &\leq |u(x)-(u)_{B(x,r)}| + |(u)_{B(x,r)}-(u)_{B(z_1,r)\cap D}|\\
 &\quad+ \dist((u)_{B(z_1,r)\cap D},\S^{N-1}).
\end{split}
 \]
\end{proof}
As an immediate corollary we obtain.
\begin{corollary}\label{co:solugeq12}
Let $\solu$ and $D$ be as in Theorem~\ref{th:main}. There exists $\epsilon_0 > 0$ such that if $B(x_0,R/2)\cap \partial D \neq \emptyset$ and \eqref{eq:distbdsmallness} holds for some $\epsilon < \epsilon_0$, then $|u| > \frac{1}{2}$ in $B(x_0,R/2)\cap D$.
\end{corollary}
As a consequence, when we reflect the maps from Theorem~\ref{th:main}, we obtain a critical equation with the growth estimates such that Proposition~\ref{pr:gensysreg} is applicable.
\begin{proposition}\label{pr:plapnup}
Let $\solu$ and $D$ be as in Theorem~\ref{th:main}. There exists $\epsilon_0 = \epsilon_0(D) > 0$ such that for any $B(x_0,4R) \subset \R^n$ on which $u$ satisfies \eqref{eq:distbdsmallness} for some $\epsilon<\epsilon_0$ there exists $v \in W^{1,p}(B(x_0,R),\R^N)$ such that 
\[
 v = u \quad \mbox{in $B(x_0,R) \cap D$},
\]
\begin{equation}\label{eq:reflectedeq}
 |\div(|\nabla \solv|^{p-2} \nabla \solv)|\aleq |\nabla v|^p \quad \text{in }B(x_0,R).
\end{equation}
Moreover, the $v$ satisfies the growth conditions from Proposition~\ref{pr:gensysreg}.
\end{proposition}

\begin{proof}
The main point is to prove that $v$ satisfies the growth conditions. The estimate \eqref{eq:reflectedeq} follows from the geometric reflection, more precisely \cite[Lemma 2.5]{SchevenDiss}. But for reader's convenience we {state the argument in full in the case where the boundary is flat. This means that we work in a ball $B(x_0,4R)$ such that $B^+(x_0,4R)\subset D \subset \R^n_+$ and $\partial D \cap B(x_0,4R)=\partial \R^n_+\cap B(x_0,4R)$.}

If $B(x_0,R) \subset \R^n_+$ then we can just take $v \equiv u$. So assume that $B(x_0,R) \cap \partial \R^n_+ \neq \emptyset$, then for $\epsilon_0$ small enough we have $|u| > \frac{1}{2}$ in $B^+(x_0,R)$ by Corollary~\ref{co:solugeq12}.

Denote by $\tilde u$ the even reflection, i.e., 
\[
 \tilde{u}(x',x_n) :=u(x',|x_n|).
\]
Moreover, set 
\[
 \sigma(q) := \frac{q}{|q|^2}, \quad q \in \R^n \backslash \{0\}.
\]
Now we define the geometric reflection $v$ as
\[
 v(x) := \begin{cases}
          u(x)\quad &\mbox{$x \in B^+(x_0,R)$}\\
          \sigma (\tilde{u}(x)) \quad&\mbox{$x \in  B(x_0,R) \backslash \R^n_+$.}
         \end{cases}
\]
Since $|u| > \frac{1}{2}$ and $u$ is uniformly bounded by Lemma~\ref{la:boundedness}, $v$ is well-defined in $B(x_0,R)$.

We also set 
\[
 \Sigma_{ij}(q) := \partial_i \sigma^j (q) = \frac{\delta_{ij} - 2\frac{q^i q^j}{|q|^2}}{|q|^2}.
\]
That is, for $x \in B(x_0, R) \backslash \R^n_+$,
\begin{equation}\label{eq:nablavnablau}
 \nabla v(x) = \Sigma(\tilde{u}(x))\ \nabla \tilde{u}(x).
\end{equation}
Observe that $\Sigma$ is symmetric, and
\[
 \Sigma(q) = \frac{1}{|q|^2} \brac{I-2 \frac{q}{|q|}\otimes \frac{q}{|q|}}
\]
and that $\frac{q}{|q|}$ is an eigenfunction to the eigenvalue $-\frac{1}{|q|^2}$, and any orthonormal basis of $\left (\frac{q}{|q|} \right )^\perp$ is the basis of the eigenspace of the eigenvalue  $\frac{1}{|q|^2}$. In particular,
\[
 |\Sigma(q) w| = \frac{1}{|q|^2} |w| \quad \forall w \in \R^N.
\]
Thus,
\begin{equation}\label{eq:nablaveqnablau}
|\nabla v(x)| = \begin{cases}
                 |\nabla \tilde{u}(x)| \quad &x \in {B^+(x_0,R)}\\
                 \frac{1}{|\tilde{u}(x)|^2}|\nabla \tilde{u}(x)| \quad &x \in {B(x_0,R)\setminus \R^n_+}.
                \end{cases}
\end{equation}

Also observe that for $|q| = 1$,
\[
 \Sigma(q) v = \Pi(q)v - \Pi^\perp(q)v \quad \mbox{for all  $v \in \R^N$},
\]
where $\Pi(q):=I-q\otimes q$ is the orthogonal projection onto $T_q \S^{N-1} =  q^\perp$ and $\Pi^\perp(q):=q\otimes q$ is the orthogonal projection onto $(T_q \S^{N-1})^\perp = \operatorname{span} \{q\}$.

Therefore, for $\varphi \in C_c^\infty(B(x_0,R),\R^N)$, since $\partial_\nu \solu \perp T_{\solu} \S^{N-1}$,
\[
 \int_{{B^+(x_0,R)}} |\nabla \solu|^{p-2} \nabla \solu \cdot \nabla \varphi + \int_{{B(x_0,R)\setminus \R^n_+}} |\nabla \tilde{\solu}|^{p-2}\nabla \tilde{\solu} \cdot \nabla (\Sigma(\tilde{u}) \varphi) = 0.
\]
In particular,
\[
 \int_{{B(x_0,R)}} |\nabla \tilde{\solu}|^{p-2} \nabla \solv \cdot \nabla \varphi  = -\int_{{B(x_0,R)\setminus \R^n_+}} |\nabla \tilde{\solu}|^{p-2} \nabla \tilde{\solu}\cdot \nabla (\Sigma(\tilde{u}))\ \varphi. 
\]
Combining this with \eqref{eq:nablaveqnablau},
\[
\begin{split}
 \int_{{B(x_0,R)}} |\nabla \solv|^{p-2} \nabla \solv \cdot \nabla (m \varphi)  &= -\int_{{B(x_0,R)\setminus \R^n_+}} |\nabla \tilde{\solu}|^{p-2} \nabla \tilde{\solu}\cdot \nabla (\Sigma(\tilde{u}))\ \varphi\\
 &\quad+ \int_{{B(x_0,R)}} |\nabla \solv|^{p-2} \nabla \solv \cdot \nabla m\ \varphi,
\end{split}
 \]
where 
\[
 m(x) = \begin{cases}
         1 \quad &\mbox{in {$B^+(x_0,R)$}}\\
         |\tilde{u}(x)|^{2(p-2)} \quad &\mbox{in {$B(x_0,R)\setminus \R^n_+$}}.\\
        \end{cases}
\]
Observe that $m(x)$ and $m(x)^{-1} \in L^\infty\cap W^{1,p}(B(x_0,R))$. Now \eqref{eq:reflectedeq} follows from \eqref{eq:nablaveqnablau}. 

It remains to establish the growth estimates from Proposition~\ref{pr:gensysreg} which follow from Proposition~\ref{pr:growth}. Indeed, set $\mathcal{M} := B(x_0,R) \cap \partial \R^n_+$. 

{To obtain \eqref{eq:growthcond} let $B(y_0,4r)\subset B(x_0,R)$ and $B(y_0,2r)\cap\mathcal{M} {=} \emptyset $. Let us consider first $B(y_0, 2r)\subset \R^n_-$. Then we observe that by \eqref{eq:nablaveqnablau} combined with the fact that $|u|>\frac12$ on $B^+(x_0,R)$ we have $\int_{B(y_0,r)}|\nabla v|^p \aleq \int_{B(\tilde{y}_0,r)} |\nabla u|^p$,
where $\tilde{y}_0$ is the point $y_0=(y_0^1,\ldots,y_0^n)$ reflected along the hyperplane $\partial \R^n_+$, i.e., $\tilde{y}_0=(y_0^1,\ldots,-y_0^n)$. Now applying \eqref{eq:growth:pllnbd0} to $u$, we obtain
\begin{equation}\label{eq:growthestimatesv}
\begin{split}
 \int_{B(y_0,r)}|\nabla v|^p &\aleq C \lambda\int_{B^+(\tilde{y}_0,4r)}|\nabla u|^p + C\lambda^{1-p}r^{-p}\int_{B^+(\tilde{y}_0,4r)}|u-(u)_{B^+(\tilde{y}_0,4r)}|^p\\
 &\le  C \lambda\int_{B(y_0,4r)}|\nabla v|^p + C\lambda^{1-p}r^{-p}\int_{B^-(y_0,4r)}\left|\tilde{u}-(\tilde{u})_{B^-({y}_0,4r)}\right|^p.
\end{split}
 \end{equation}
To estimate the remaining part we note that since $v=\frac{\tilde{u}}{|\tilde{u}|^2}$ we have $\tilde{u} = \frac{v}{|v|^2}$ in $\R^n_{-}$ { and for any $A\subset B(x_0,R)\setminus \R^n_+$: }

\begin{equation}\label{eq:growthestimateforv1}
 \begin{split}
  \mvint_A \left|\frac{v}{|v|^2} -\brac{\frac{v}{|v|^2}}_A \right|^p & {\aleq }\mvint_A \mvint_A \left|\frac{v(x)}{|v(x)|^2} - \frac{v(y)}{|v(x)|^2}\right|^p + \mvint_A \mvint_A \left|\frac{v(y)}{|v(x)|^2} - \frac{v(y)}{|v(y)|^2}\right|^p \\
  & \aleq \|v^{-1}\|^{2p}_{L^{\infty}} \mvint_A \mvint_A |v(x)-v(y)|^p + \|v^{-1}\|^{{3p}}_{L^\infty}\mvint_A \mvint_A \left|{|v(x)|^2-|v(y)|^2}\right|^p.
 \end{split}
\end{equation}
Now, since for any $a,\, b $,
\[
 |a|^2 - |b|^2 = (|a|+|b|)(|a| - |b|) \le (|a|+|b|)|a-b|
\]
we have 
\begin{equation}\label{eq:growthestimateforv2}
\begin{split}
 \mvint_A \mvint_A \left|{|v(x)|^2-|v(y)|^2}\right|^p &\aleq \|v\|^{{p}}_{L^\infty(A)}\mvint_A \mvint_A |v(x) - v(y)|^p\\
 &\aleq \|v\|^{{p}}_{L^\infty(A)}\mvint_A |v-(v)_A|^p,
 \end{split}
\end{equation}
where the last inequality was obtained by adding and subtracting $(v)_A$ and by the triangle inequality. We deduce from \eqref{eq:growthestimateforv1} and \eqref{eq:growthestimateforv2} that
\[
  \mvint_A \left|\frac{v}{|v|^2} -\brac{\frac{v}{|v|^2}}_A \right|^p \aleq \|v^{-1}\|^{2p}_{L^\infty(A)}(1+{\|v\|^{p}_{L^\infty(A)}\|v^{-1}\|^{p}_{L^\infty(A)}})\mvint_A |v-(v)_A|^p.
\]
Due to the fact that $|u|>\frac12$ and $u$ is uniformly bounded we get
\begin{equation}\label{eq:growthestimatesvmiddle}
  \mvint_A |\tilde{u} -(\tilde{u})_A |^p \aleq \mvint_A |v - (v)_A|^{p} \quad \text{for any } A\subset B(x_0,R)\setminus\R^n_{+}.
\end{equation}}
To conclude, we note\footnote{\label{ft:AB} {Indeed, for any $\tilde{A}\subset A$ we have by enlarging the domain of integration and applying Jensen's inequality
\[                                                                                                                                                                                                                                                                                                                                                                                                                                                                                                                                                                      
 \mvint_{\tilde{A}} |w-(w)_{\tilde{A}}|^p \aleq \frac{|A|}{|\tilde{A}|} \mvint_{{A}} |w-(w)_{{A}}|^p.
 \]
}} {that since $B(y_0,2r) \subset \R^n_-$ we have $\frac{|B(y_0,4r)|}{|B^-(y_0,4r)|}\aeq 1$, thus}
\begin{equation}\label{eq:growthestimatesvfinal}
 \mvint_{{B^-(y_0,4r)}} |v-(v)_{{B^-(y_0,4r)}}|^p \aleq  \mvint_{B(y_0,4r)} |v-(v)_{B(y_0,4r)}|^p.
\end{equation}
Combining estimates \eqref{eq:growthestimatesv}, \eqref{eq:growthestimatesvmiddle}, and \eqref{eq:growthestimatesvfinal} we obtain \eqref{eq:growthcond}.
The second case $B(y_0,2r)\subset\R^n_+$ is easier and we leave it to the reader.

Finally, for \eqref{eq:growthcondinter} we {apply \eqref{eq:growth:pllnbd} and} observe that $|u|^2 \equiv 1$ on $\mathcal{I} := B(y_0,4r) \cap \partial \R^n_+$. Thus,
\[
 \int_{B^+(y_0,4r)} \big ||\solu|^2-1 \big |^p \aleq \brac{\|u\|_{L^\infty}+1} \int_{B^+(y_0,4r)} \big ||\solu|-(|u|)_{\mathcal{I}} \big|^p.
\]
Now
\[
 \big ||\solu(z)|-(|u|)_{\mathcal{I}} \big | \leq \mvint_{\mathcal{I}} \big | |\solu(z)|-|u(z_2)| \big |\ dz_2 \leq \mvint_{\mathcal{I}} \big | \solu(z)- u(z_2) \big |\ dz_2
\]
and thus
\[
 \mvint_{B^+(y_0,4r)} \big ||\solu|-(|u|)_{\mathcal{I}} \big |^p \aleq \mvint_{B^+(y_0,4r)} \big |\solu -(u)_{\mathcal{I}} \big |^p + \mvint_{\mathcal{I}} \big |\solu -(u)_{\mathcal{I}}\big |^p.
\]
Proposition~\ref{pr:plapnup} is now established.
\end{proof}

\begin{proof}[Proof of Theorem~\ref{th:main}]
For $p=n$ H\"older continuity for $u$ follows from Proposition~\ref{pr:main:hoelderpn}. For $p<n$ it follows from the combination of Proposition~\ref{pr:plapnup} and Proposition~\ref{pr:gensysreg}. Now $C^{1,\alpha}$-regularity follows from the reflection, Proposition~\ref{pr:plapnup}, and the fact that a H\"older continuous solution to the reflected system is $C^{1,\alpha}$ for some $\alpha > 0$, see \cite[Theorem 3.1.]{Hardt-Lin-1987} (which is stated for minimizers but actually only uses the continuity of the solution and the equation). See also \cite[Theorem 1.2.]{Riviere-Strzelecki-2005}.

Note that for $p=n$ there is also a more elegant argument to pass from $C^\alpha$ regularity to $C^{1,\alpha}$. Testing the equation \eqref{eq:soluspheredef} in $x$ and $x+h$ with $\varphi(x) := \eta(x) (v(x+h)-v(x))$ for a suitable cutoff function $\eta$ one obtains from the H\"older continuity of $u$ that for some $\sigma > 0$ we have $\nabla v \in W^{1+\sigma,n}$. In particular, by Sobolev embedding $\nabla v \in L^{(n,1)}_{loc}$, and from Duzaar-Mingione's work \cite{Duzaar-Mingione-2010} we get a Lipschitz bound for $v$. Now, $C^{1,\alpha}$-regularity is a consequence of the potential estimates for $p$-Laplace equations, see \cite{Kuusi-Mingione-2012,Kuusi-Mingione-2017}. We leave the details to the reader.
\end{proof}

\section{Partial Regularity: Proof of Theorem \ref{th:partialregularity}}\label{s:partialregularity}
For simplicity we assume {in this section that $B^+(0,R)\subset D \subset \R^n_+$ and $\partial D \cap B(0,R)= \partial \R^n_+\cap B(0,R)$}. We begin with recalling that a map $u\in W^{1,p}({B^+(0,R)},\R^N)$ is said to be \emph{stationary $p$-harmonic} with respect to the free boundary condition $u({\partial D \cap B(0,R)})\subset\S^{N-1}$ if in addition to \eqref{eq:soludef} it is a critical point of the energy with respect to variations in the domain. The latter is equivalent to $u$ satisfying
\begin{equation}\label{eq:firstvariation}
 \int_{{B^+(0,R)}}|\nabla u|^{p-2}\brac{|\nabla u|^2 \delta_{ij} - p\,\partial_i u \,\partial_j u}\partial_i\xi^j  =0
\end{equation}
for $\xi = (\xi^1,\ldots,\xi^n)\in C_c^\infty({\overline{\R^n_+}\cap B(0,R)},\R^n)$ with $\xi(\partial \R^n_+)\subset\partial\R^n_+$.  

By choosing the test function as $\xi(x):=\psi(x)(x_0-x)$ in \eqref{eq:firstvariation}, where $\psi\in C^\infty_c({\overline{\R^n_+}\cap B(0,R)},[0,1])$ is a suitable bump function, one obtains the following.

\begin{lemma}[monotonicity formula]\label{lem:monotonicityformula}
Let $u\in W^{1,p}(B^+(0,R),\R^N)$ be a stationary $p$-harmonic map with respect to the free boundary condition $u( B^+(0,R)\cap\{x_n=0\})\subset\S^{N-1}$ and let $x_0 \in B^+(0,R)\cap \{x_n=0\}$. Then, the normalized $p$-energy is monotone. In particular,
\begin{equation}\label{eq:mono}
 r^{p-n}\int_{B^+(x_0,r)}|\nabla u|^p -\rho^{p-n}\int_{B^+(x_0,\rho)}|\nabla u|^p = p\int_{B^+(x_0,r)\setminus B^+(x_0,\rho)}|x-x_0|^{p-n} {|\nabla u|^{p-2}}\, \left|\frac{\partial u}{\partial \nu}\right|^2
\end{equation}
for all $0<\rho<r<R-|x_0|$, where $\nu$ is the outward pointing unit normal for $\partial B(x_0,r)$, $\nu(x):=\frac{x-x_0}{|x-x_0|}$. For $x_0\in B^+(0,R)\setminus\partial\R^n_+$ the same holds if $r$ is such that $B^+(x_0,r) = B(x_0,r)\subset\R^n_+$.
\end{lemma}

This well-known fact was proved for Yang--Mills fields and stationary harmonic maps by Price \cite{Price}, see \cite{Evans1991,Bethuel1993} and also \cite[Section 2.4]{Simon}. Fuchs \cite{Fuchs} observed that \eqref{eq:mono} holds for stationary $p$-harmonic maps. As pointed out by Scheven \cite[p.137]{Scheven-2006} the proof holds true in the case of free boundary condition.

We will need the following lemma (see, e.g., \cite[Corollary 3.2.3.]{Ziemer}).
\begin{lemma}[Frostman's lemma]\label{le:frostman}
If $f\in L^{p}(\R^n)$, $p\ge 1$, and $0\le\alpha<n$, then for 
\[
 E := \left\{x\in \R^n:\limsup_{r\rightarrow 0} r^{-\alpha}\int_{B(x,r)}|f(y)|^p >0 \right\},
\]
we have $\H^{\alpha}(E)=0$.
\end{lemma}
 We shall show, using monotonicity formula \eqref{eq:mono} and Frostman's Lemma \ref{le:frostman}, that the set outside which the condition \eqref{eq:thmain:smallnesscond} is satisfied is of zero $(n-p)$-Hausdorff measure. We then obtain Theorem \ref{th:partialregularity} from Theorem \ref{th:main}.
\begin{proof}[Proof of Theorem~\ref{th:partialregularity}]
Let 
 \begin{equation*}
  S := \left\{ x\in \overline{\R^n_+} : \limsup_{r\rightarrow 0} r^{p-n}\int_{B^+(x,r)}|\nabla u|^p >0 \right\}, 
 \end{equation*}
by Lemma \ref{le:frostman}, we have $\H^{n-p}(S)=0$.
 
We define for $\epsilon$ as in Theorem \ref{th:main}
\[
  \Sigma_{\epsilon} : = \left\{x \in \overline{\R^n_+} : \forall R>0\sup_{|y_0 - x|<R}\sup_{\rho<R}\rho^{p-n}\int_{B^+(y_0,\rho)}|\nabla u|^p \ge \epsilon \right\},
 \]
 clearly $\Sigma_{\epsilon}$ is a closed set. We will prove that $\H^{n-p}(\Sigma_\epsilon)=0$. Then Theorem \ref{th:partialregularity} is a consequence of Theorem \ref{th:main}. 
 
Let $A_\epsilon$ be the set on which the condition \eqref{eq:thmain:smallnesscond} is satisfied for $\epsilon$, i.e., 
\[
 A_{\epsilon} := \overline{\R^n_+}\setminus\Sigma_\epsilon = \left\{x\in \overline{\R^n_+}: \exists R>0 \mbox{ such that } \sup_{|y_0 - x|<R}\sup_{\rho<R}\rho^{p-n}\int_{B^+(y_0,\rho)}|\nabla u|^p <\epsilon \right\}.
\]
In order to prove the theorem it suffices to show that $\brac{\overline{\R^n_+}\setminus S}\subseteq A_\epsilon$. 
 
{\phantom{Let $\epsilon>0$ be as in Theorem \ref{th:main} and }}Let $x_0 \in \brac{\overline{\R^n_+}\setminus S}$, i.e., be such that $\limsup_{r\rightarrow 0} r^{p-n}\int_{B^+(x_0,r)}|\nabla u|^p =0$. There exists an $R>0$ such that 
\[ R^{p-n}\int_{B^+(x_0,R)}|\nabla u|^p<4^{p-n}\epsilon.\] 
We shall show that 
\[
\sup_{|y_0 - x_0|<R/4}\sup_{\rho\, < R/4} \rho^{p-n}\int_{B^+(y_0, \rho)}|\nabla u|^p <  \epsilon.
\]
Choose any $y_0$ such that $|y_0-x_0|<R/4$ and any radius $\rho<R/4$. First observe that we may take $y_0\in \overline{\R^n_+}$. {Indeed, suppose that $y_1\in B(x_0,R/4)\cap\R^n_-$, then for any $\rho<R/4$ we can choose $y_0\in B(x_0,R/4)\cap\overline{\R^n_+}$ such that $B(y_1,\rho)\cap\overline{\R^n_+}\subset B(y_0,\rho)\cap \overline{\R^n_+}$ thus}
\[
\sup_{|y_1 - x_0|<R/4}\sup_{\rho\, < R/4} \rho^{p-n}\int_{B^+(y_1, \rho)}|\nabla u|^p = \sup_{y_0\in B(x_0,R/4)\cap\overline{\R^n_+}}\sup_{\rho\, < R/4} \rho^{p-n}\int_{B^+(y_0, \rho)}|\nabla u|^p.
\]
Now assume that $y_0\in\partial\R^n_+$. We have $B^+(y_0,\rho)\subset B^+(y_0,R/4)\subset B^+(x_0,R)$. Thus
\[
 \rho^{p-n}\int_{B^+(y_0,\rho)}|\nabla u|^p \le \brac{\frac R4}^{p-n}\int_{B^+(y_0,R/4)}|\nabla u|^p \le 4^{n-p} R^{p-n}\int_{B^+(x_0,R)}|\nabla u|^p < \epsilon,
\]
where the first inequality is a consequence of the monotonicity formula \eqref{eq:mono}. 

Now, let us assume that $y_0\notin\partial\R^n_+$. Let $\overline{\rho} = \dist(y_0,\partial\R^n_+)$ and $\overline{y}_0$ be the projection of $y_0$ onto $\partial\R^n_+$. We can assume that $\rho<\overline{\rho}$. Indeed, if not we would have 
\begin{align*}
 \rho^{p-n}\int_{B^+(y_0,\rho)}|\nabla u|^p &\le \rho^{p-n}\int_{B^+(\overline{y}_0,2\rho)}|\nabla u|^p = 2^{n-p}(2\rho)^{p-n}\int_{B^+(\overline{y}_0,2\rho)}|\nabla u|^p\\
 &\le 2^{n-p}\brac{\frac R2}^{p-n}\int_{B^+(\overline{y}_0,R/2)}|\nabla u|^p\le 4^{n-p}R^{p-n}\int_{B^+(x_0,R)}|\nabla u|^p < \epsilon.
\end{align*}
Next, we note that $\overline{\rho}<R/4$ and
observe the following inclusions 
\[
B(y_0,\rho)\subset B(y_0,\overline{\rho})\subset B^+(\overline{y}_0,2\overline{\rho})\subset B^+(\overline{y}_0,R/2)\subset B^+(x_0,R)\] 
and the following inequalities which are consequences of the monotonicity formula \eqref{eq:mono}:
\begin{align*}
 \rho^{p-n}\int_{B(y_0,\rho)} |\nabla u|^p &\le \brac{\overline{\rho}}^{p-n}\int_{B(y_0,\overline{\rho})} |\nabla u|^p,\\
 (2\overline{\rho})^{p-n}\int_{B^+(\overline{y}_0,2\overline{\rho})}|\nabla u|^p &\le \brac{\frac R2}^{p-n}\int_{B^+(\overline{y}_0,R/2)} |\nabla u|^p. 
\end{align*}
Thus
\begin{align*}
 \rho^{p-n}\int_{B(y_0,\rho)} |\nabla u|^p &\le \brac{\overline{\rho}}^{p-n}\int_{B(y_0,\overline{\rho})} |\nabla u|^p \le 2^{n-p} (2\overline{\rho})^{p-n}\int_{B^+(\overline{y}_0,2\overline{\rho})}|\nabla u|^p\\ 
 &\le 2^{n-p} \brac{\frac R2}^{p-n}\int_{B^+(\overline{y}_0,R/2)}|\nabla u|^p\le 4^{n-p} R^{p-n}\int_{B^+(x_0,R)}|\nabla u|^p < \epsilon,
\end{align*}
which gives $x_0\in A_\epsilon$.

We conclude $\Sigma_\epsilon\subset S$ and thus $\H^{n-p}(\Sigma_\epsilon)=0$.
\end{proof}

\subsection{A Liouville type result}
{We note that the monotonicity formula in Lemma~\ref{lem:monotonicityformula} can be used to prove partial regularity but also Liouville type results in the spirit of \cite{Liu_2010}. Indeed, if we work in $\R^n_+$, for $u\in \dot{W}^{1,p}(\R^n_+,\R^N)$ we can say that $u$ is stationary $p$-harmonic with respect to the free boundary condition $u(\p \R^n_+)\subset \mathbb{S}^{N-1}$ if $u$ satisfies \eqref{eq:soludef} and \begin{equation}\label{eq:firstvariation2}
 \int_{\R^n_+}|\nabla u|^{p-2}\brac{|\nabla u|^2 \delta_{ij} - p\,\partial_i u \,\partial_j u}\partial_i\xi^j  =0
\end{equation}
for $\xi = (\xi^1,\ldots,\xi^n)\in C_c^\infty(\overline{\R^n_+},\R^n)$ with $\xi(\partial \R^n_+)\subset\partial\R^n_+$. We then have

\begin{proposition}\label{pr:Liouvilletype}
Let $2\le p<n$ and $u\in \dot{W}^{1,p}(\R^n_+,\R^N)$ be such that $u$ is a finite energy, stationary $p$-harmonic map with respect to the free boundary condition $u(\p \R^n_+)\subset \mathbb{S}^{N-1}$, then $u$ is constant.
\end{proposition}

\begin{proof}
By contradiction, assume that $u$ is not a constant. Then there exists $R_0>0$ such that $\int_{B^+(0,R_0)}|\nabla u|^p \geq c>0$. Now by the monotonicity formula \ref{lem:monotonicityformula} we have that for any $R>R_0$
\begin{equation}
\int_{B^+(0,R)} |\nabla u|^p\geq \left(\frac{R}{R_0}\right)^{n-p} \int_{B^+(0,R_0)}|\nabla u|^p\geq \left(\frac{R}{R_0}\right)^{n-p} c.
\end{equation}
We can then let $R$ go to $+\infty$ and we obtain that the $p$-energy of $u$ in $\R^n_+$ is infinite. This is a contradiction since we assumed that $u\in \dot{W}^{1,p}(\R^n_+,\R^N)$.
\end{proof}

}
\subsection*{Acknowledgments}
A.S. and K.M. are supported by the German Research Foundation (DFG) through grant no.~SCHI-1257-3-1. A.S. received research funding from the Daimler and Benz foundation no. 32-11/16, and Simons foundation through grant no 579261 is gratefully acknowledged. A.S. was Heisenberg fellow. R.R. was supported by the Millennium Nucleus Center for Analysis of PDE NC130017 of the Chilean Ministry of Economy and by the F.R.S.-FNRS under the ``Mandat d'Impulsion scientifique F.4523.17, Topological singularities of Sobolev maps". {We thank M. Willem for indicating us the proof of Proposition \ref{th:boundedness-in-unbounded}. The authors would like to thank the anonymous referees for their helpful suggestions.}

\appendix
\section{On boundedness of \texorpdfstring{$p$}{p}-harmonic maps}
The following lemma is well-known. However, we could not find it explicitly in the literature, so we state it here for the convenience of the reader. 
\begin{lemma}\label{la:boundedness}
Let $D \subset \R^n$ be a smooth, bounded domain. Assume that $\solu \in W^{1,p}(D,\R^N)$ is a solution to 
\[
\dv (|\nabla \solu|^{p-2} \nabla \solu) = 0 \quad \mbox{in $D$}.
\]
If $\solu \Big |_{\partial D} \in L^\infty({\partial} D)$, then $\|\solu\|_{L^{\infty}(D)} \leq  \|\solu\|_{L^{\infty}(\partial D)}$.
\end{lemma}
\begin{proof}
For scalar functions this is a consequence of the weak maximum principle for the $p$-Laplacian, see \cite[Theorem 2.15.]{Lindqvist-2006}. However, here we work with a system. For $\eps \in (0,1)$ we find smooth solutions $\solu_\eps \in W^{1,p}\cap C^\infty(D,\R^N)$ of the uniformly elliptic system
\begin{equation}\label{eq:soleps}
\begin{cases}
\dv (\brac{\eps + |\nabla \solu_\eps|^2}^{\frac{p-2}{2}} \nabla \solu_\eps) = 0 \quad &\mbox{in $D$}\\
\solu_\eps = \solu \quad &\mbox{on $\partial D$}\\
\end{cases}
\end{equation}
The solution is smooth in the interior, and a direct computation shows that
\begin{equation}
\dv (\brac{\eps + |\nabla \solu_\eps|^2}^{\frac{p-2}{2}} \nabla |\solu_\eps|^2) \geq 0.
\end{equation}
Thus the weak maximum principle for scalar solutions of uniformly elliptic operators in divergence form implies
\begin{equation}\label{eq:soluepsinfty}
 \sup_{\eps \in (0,1)} \|\solu_\eps\|_{L^{\infty}(D)} \leq \|\solu\|_{L^{\infty}(\partial D)},
\end{equation}
Moreover, we can test \eqref{eq:soleps} with $\solu_\eps - \solu$, which is trivial on $\partial D$, and thus
\[
 \int_{D} |\nabla \solu_\eps|^{p} \leq \int_{D} \brac{\eps + |\nabla \solu_\eps|^2}^{\frac{p-2}{2}} |\nabla \solu_\eps|^2= \int_{D} \brac{\eps + |\nabla \solu_\eps|^2}^{\frac{p-2}{2}} \nabla \solu_\eps\ \cdot \nabla \solu,
\]
consequently, with Young's inequality,
\[
 \int_{D} |\nabla \solu_\eps|^{p} \leq \frac{1}{2} \int_{D} |\nabla \solu_\eps|^p + C\int_{D} |\nabla \solu|^p + C(|D|,p)
\]
Thus, $\solu_\eps$ is uniformly bounded in $W^{1,p}$,
\begin{equation}\label{eq:epsboundedness}
 \sup_{\eps \in (0,1)} \int_{D} |\nabla \solu_\eps|^{p} < \infty.
\end{equation}
On the other hand,
\[
 \int_{D} \brac{\brac{\eps + |\nabla \solu_\eps|^2}^{\frac{p-2}{2}} \nabla \solu_\eps - |\nabla \solu|^{p-2} \nabla \solu} \cdot (\nabla \solu_\eps- \nabla \solu) = 0.
\]
Applying then the well-known inequality
\[
 |a-b|^p \aleq \brac{|a|^{p-2} a - |b|^{p-2} b}(a-b),
\]
we find that as $\eps \to 0$,
\[
 \int_D |\nabla \solu - \nabla \solu_\eps|^{p} \aleq o(1) \int_{D} \brac{|\nabla \solu|^{p-1} + |\nabla \solu_\eps|^{p-1}} 
\]
Therefore, in view of \eqref{eq:epsboundedness} and the boundedness of $D$,
\[
 \solu_\eps \xrightarrow{\eps \to0} \solu \quad \mbox{in $W^{1,p}(D)$}.
\]
In particular, up to a subsequence, we have pointwise almost everywhere convergence, and from \eqref{eq:soluepsinfty} we have
\[
 \|\solu \|_{L^{\infty}(D)} \leq \|\solu \|_{L^{\infty}( \partial D)} .
\]
%
\end{proof}

{
\begin{lemma}\label{la:boundedness2}
Let $D \subset \R^n$ be a possibly unbounded domain with smooth boundary $\partial D$. Assume that $p > n-1$, $\solu \in \dot{W}^{1,p}(D,\R^N)$ is a solution to 
\[
\dv (|\nabla \solu|^{p-2} \nabla \solu) = 0 \quad \mbox{in $D$}.
\]
If $\solu \Big |_{\partial D} \in L^\infty({\partial} D)$, then for every compact set $K \subset \overline{D}$ we have 
\[
\|u\|_{L^\infty(K)} < \infty.
\]
\end{lemma}
\begin{proof}
For compact $K$ we find by Fubini's theorem a smooth, bounded domain $\tilde{D} \supset K$ such that 
\[
u \Big |_{\partial \tilde{D} \cap D} \in W^{1,p},
\]
Since $p > n-1$ we conclude that, by Morrey-Sobolev embedding, $u$ is continuous on $\partial \tilde{D} \cap D$, and in particular $u \in L^\infty(\partial \tilde{D})$. Now we can apply Lemma~\ref{la:boundedness} to $\tilde{D}$ to obtain the result.
\end{proof}}

We now prove a maximum principle analog of Lemma~\ref{la:boundedness} but for maps defined in the half-space $\R^n_+$. We work with maps with finite energy, i.e., we work with $\dot{W}^{1,p}(\R^n_+,\R^N):=\{v \in \mathcal{D}'(\R^n_+,\R^N); \nabla v \in L^p(\R^n_+,\R^N) \}$. We remark that a map in $ \dot{W}^{1,p}(\R^n_+,\R^N)$ is also in $L^p_{\text{loc}}(\R^n_+,\R^N)$ and hence has a trace on $\p \R^n_+:= \R^{n-1} \times \{0 \}$ which is well-defined.

{
\begin{proposition}\label{th:boundedness-in-unbounded}
Let $u\in \dot{W}^{1,p}(\R^n_+,\R^N)$ be a solution to 
\[
 \divv ( |\nabla u|^{p-2}\nabla u)=0 \quad \text{ in } \R^n_+ ,
\]
that is
\[
\int_{\R^n_+} |\nabla u|^{p-2} \nabla u \cdot \nabla \varphi = 0 \quad \mbox{for all $\varphi \in C_c^\infty(\overline{\R^n_+})$}.
\]
Assume that $u \big\rvert_{\R^{n-1} \times \{0\}} \in L^\infty(\R^{n-1}\times \{0 \})$, then $u\in L^\infty(\R^n_+)$ and 
$$\|u\|_{L^\infty(\R^n_+)}\leq \|u\|_{L^\infty(\p \R^n_+)}. $$
\end{proposition}

\begin{proof}
We denote by $g:=u \big\rvert_{\R^{n-1} \times \{0\}}$ and $M:=\|g\|_{L^\infty(\p \R^n_+)}$. From Proposition \ref{prop:uniqueness} below we know that $u$ is the unique minimizer of the energy $\int_{\R^n_+}|\nabla v|^p$ in $X:= \big\{v \in \dot {W}^{1,p}(\R^n_+,\R^N)\colon v \big\rvert_{\R^{n-1} \times 0}=g \text{ in the trace sense}  \big\}$. Now we define 
\[ 
\tilde{u}:= 
 \left\{ \begin{array}{ll}
    u & \textrm{ if } |u|\leq M,\\
    \frac{Mu}{|u|} & \textrm{ if } |u|> M.
  \end{array} \right.
\]
By a direct computation we can see
\[
\int_{\R^n_+} |\nabla \tilde{u}|^p \leq \int_{\R^n_+}|\nabla u|^p.
\]
Besides we have $\tilde{u}\big\rvert_{\p R^n_+}=g$. Thus by uniqueness we deduce that $\tilde{u}=u$ and $|u|\leq M$ in $\R^n_+$. This concludes the proof.
\end{proof}

It remains to prove:

\begin{proposition}\label{prop:uniqueness}
Let $u\in \dot{W}^{1,p}(\R^n_+,\R^N)$ be as in Proposition~\ref{th:boundedness-in-unbounded} a solution to 
\[
\divv (|\nabla u|^{p-2}\nabla u)=0 \quad \text{ in } \R^n_+. 
\]
Let us denote by $g=u \big\rvert_{\R^{n-1} \times 0}$ the trace of $u$. Then $u$ is the unique minimizer of the energy $\int_{\R^n_+}|\nabla v|^p$ in 
\[
 X:=\left\{v \in \dot {W}^{1,p}(\R^n_+,\R^N)\colon  v \big\rvert_{\R^{n-1} \times 0}=g \text{ in the trace sense}  \right\}.
\]
\end{proposition}

\begin{proof}
By the direct method of calculus of variations we can prove that there exists a minimizer $u_0$ of $\int_{\R^n_+}|\nabla u|^p$ in $X$. Besides, by strict convexity of the $p$-energy we have that this minimizer is unique and it is the unique critical point of the $p$-energy in $X$. 
That is there is at most one map with a trace equal to $g$ which satisfies 
\begin{equation}\label{eq:critical-point1}
\int_{\R^n_+}|\nabla u_0|^{p-2}\nabla u_0\cdot \nabla \phi=0, \ \ \ \ \forall \phi \in \dot{W}^{1,p}(\R^n_+,\R^N), \quad \phi \big\rvert_{\R^{n-1} \times \{0\}} = 0.
\end{equation}
Observe that $C_c^\infty(\R^n_+,\R^N)$ is dense in the space
\[
Y := \left \{\phi \in \dot{W}^{1,p}(\R^n_+,\R^N)\colon \phi \big |_{\R^{n-1} \times \{0\}} = 0 \right \},
\]
which can be proven as in, e.g., \cite[Proposition 6.2.5]{Willem_2013}. 
We conclude that there is at most  one map with a trace equal to $g$ which satisfies 
\begin{equation}\label{eq:critical-point2}
\int_{\R^n_+}|\nabla u_0|^{p-2}\nabla u_0\cdot \nabla \phi=0, \ \ \ \ \forall \phi \in C_c^\infty(\R^n_+).
\end{equation}
This implies the claim.
\end{proof} 
}
\bibliographystyle{abbrv}
\bibliography{bib}

\end{document}